\documentclass[12pt]{article}
\usepackage{amsmath,amssymb, amsthm}
\usepackage[all, arc]{xy}
\usepackage{url}

\newtheorem{theorem}{Theorem}[section]
\newtheorem{lemma}[theorem]{Lemma}
\newtheorem{propo}[theorem]{Proposition}

\theoremstyle{plain}

\theoremstyle{plain}

\theoremstyle{definition}
\newtheorem{defn}[theorem]{Definition}
\theoremstyle{remark}
\newtheorem{rem}[theorem]{Remark}
\theoremstyle{definition}
\newtheorem{example}[theorem]{Example}
\renewenvironment{equation}{\medskip\noindent\refstepcounter{subsection}\makebox[0pt][l]{({\bf\thesubsection})}\begin{minipage}[b]{\textwidth}$$}{$$\end{minipage}\medskip\noindent}
\renewenvironment{split}{\begin{array}[c]{r}}{\end{array}}

\newcommand{\D}{{\mathcal D}}
\newcommand{\T}{T^{*}\mathbb{P}^{1}}
\newcommand{\Hom}{{\rm Hom}}
\newcommand{\Ext}{{\rm Ext}}
\newcommand{\End}{{\rm End}}
\newcommand{\Homg}{{\rm Hom_{G}}}
\newcommand{\Extg}{{\rm Ext}_{G}}

\newcommand{\E}{{\mathcal E}}
\newcommand{\F}{{\mathcal F}}
\newcommand{\G}{{\mathcal G}}

\newcommand{\Coh}{{\rm Coh}}

\newcommand{\bull}{^{\mbox{\tiny{$\bullet$}}}}

\begin{document}
\author{Christopher Brav}
\title{The projective McKay correspondence}
\maketitle
\begin{abstract}
Kirillov \cite{kir} has described a McKay correspondence for finite subgroups of $PSL_{2}({\mathbb C})$ that
associates to each `height' function an affine Dynkin quiver together with a derived equivalence between equivariant sheaves on $\mathbb{P}^{1}$ and representations of this quiver. The equivalences for different height functions are then related by reflection functors for quiver representations. 

The main goal of this paper is to develop an analogous story for the cotangent bundle of $\mathbb{P}^1$. We show that each height function gives rise to a derived equivalence between equivariant sheaves on the cotangent bundle $\T$ and modules over the preprojective algebra of an affine Dynkin quiver. These different equivalences are related by spherical twists, which take the place of the reflection functors for $\mathbb{P}^{1}$. 
\end{abstract}

\section{Introduction}\label{sectintroduction}

In \cite{mck}, John McKay associated to a finite subgroup $G \subset SL_{2}(\mathbb{C})$ a graph $\Gamma$ in which the set of vertices I is labeled by the irreducible representations $W_{i}$, $i \in I$ of $G$ and the number of edges $n_{ij}$ between two irreducible representations $W_{i}, W_{j}$ is given by $n_{ij}={\rm dim}\,\Homg(W_{i},V\otimes W_{j})$, where $V$ is the standard two dimensional representation of $G$ coming from its embedding in $SL_{2}(\mathbb{C})$. McKay then observed that the graph $\Gamma$ is an affine Dynkin diagram of type $A$,$D$, or $E$.

As we recall in Section~\ref{affinecase}, this relation between the representation theory of finite subgroups of $SL_{2}(\mathbb{C})$ and affine Dynkin diagrams has a description in terms of $G$-equivariant sheaves on $\mathbb{C}^{2}$. More precisely, there is an equivalence  

\begin{equation}\label{affinecorrespondence}
{\rm Coh}_{G}(\mathbb{C}^{2}) \simeq  \Pi_{\Gamma}\mbox{-mod}
\end{equation}
between the category of $G$-equivariant coherent sheaves on $\mathbb{C}^{2}$ and the category of finitely generated modules
over the preprojective algebra $\Pi_{\Gamma}$. For the purposes of this paper we shall refer to this equivalence as the McKay correspondence for $\mathbb{C}^{2}$. 

More geometrically, Kapranov-Vasserot \cite{kv}, building on work of Gonzalez-Sprinberg-Verdier \cite{verdier}, construct a derived
equivalence $$D^{b}_{G}(\mathbb{C}^{2}) \simeq D^{b}(\widehat{X})$$
where $\widehat{X} \rightarrow \mathbb{C}^{2}/\!\!/ G$ is the minimal resolution of the `Kleinian singularity'  $\mathbb{C}^{2}/\!\!/ G$.
It is this equivalence that usually goes under the name `McKay correspondence'. 

\medskip

In another direction, Kirillov \cite{kir} has described a projective McKay correspondence for finite subgroups $\widetilde{G}$ of $PSL_{2}(\mathbb{C})$. Letting $\Gamma$ be the graph associated by McKay to the double cover $G \subset SL_{2}(\mathbb{C})$ of $\widetilde{G}$, this projective correspondence relates equivariant sheaves
on $\mathbb{P}^1$ to representations of the path algebra of a quiver with underlying graph $\Gamma$. More precisely, to each `height' function 
$$h: \Gamma \rightarrow \mathbb{Z}$$ 
on the vertices of $\Gamma$ (defined in Section~\ref{sectprojmckay}), Kirillov associates a quiver $Q_{h}$ on $\Gamma$ and an exact equivalence

\begin{equation*}
\xymatrix{
D^{b}_{\widetilde{G}}(\mathbb{P}^1) \ar[rr]^{R\Phi_{h}} && D^{b}(Q_{h})
} 
\end{equation*}
where $D^{b}_{\widetilde{G}}(\mathbb{P}^1)$ is the bounded derived category of $\widetilde{G}$-equivariant coherent sheaves on $\mathbb{P}^{1}$ and $D^{b}(Q_{h})$ is the bounded derived category of representations of $Q_{h}$. Furthermore, the equivalences
for different height functions $h$ and $\tilde{h}$ differ by a sequence of the reflection functors of Bernstein-Gelfand-Ponamarev 
\cite{bgp} in the sense that there is a commutative diagram of equivalences

\begin{equation*}
\xymatrix{& D^{b}_{\widetilde{G}}(\mathbb{P}^{1}) \ar[dl]_{R\Phi_{h}} \ar[dr]^{R\Phi_{\tilde{h}}}\\
D^{b}(Q_{h}) \ar[rr]^{BGP} && D^{b}(Q_{\tilde{h}}).
}
\end{equation*}
It is well-known that the Grothendieck groups of the various quivers $Q_{h}$ can be identified with the affine root lattice associated to the diagram $\Gamma$, and that under this identification, the reflection functors generate the action of the affine Weyl group.

\medskip

The main goal of this paper is to develop an analogous story for the cotangent bundle $\T$. Theorem~\ref{cotan}, together with Proposition~\ref{cokoszul}, gives for each height function $h$ an equivalence 
\begin{equation*}
\xymatrix{
D^{b}_{\widetilde{G}}(\T) \ar[rr]^{R\Psi_{h}} && D^{b}(\Pi_{\Gamma}), 
}
\end{equation*}
where $\Pi_{\Gamma}$ is the preprojective algebra of the diagram $\Gamma$. 

In order to relate the various equivalences $R\Psi_{h}$, we consider for each $h$ a `$\Gamma$-configuration' of spherical objects 
$\E_{i}^{h}$, $i \in I$, together with the associated spherical twists of Seidel-Thomas \cite{st} which act as autoequivalences on the derived category. Just as the equivalences $R\Phi_{h}$ in the $\mathbb{P}^{1}$-case differed by reflection functors, Theorem~\ref{twist} explains how the equivalences $R\Psi_{h}$ differ by spherical twists.

To make the analogy between spherical twists and reflection functors more precise, Proposition~\ref{projtilt} reinterprets 
the latter purely in terms of $D^{b}_{\widetilde{G}}(\mathbb{P}^{1})$. Under the inverse equivalence $R\Phi_{h}^{-1}$, the heart of the standard $t$-structure on $D^{b}(Q_{h})$ is sent to a heart $\mathcal{A}_{h} \subset D^{b}_{\widetilde{G}}(\mathbb{P}^1)$ with simple objects $E_{i}^{h}$, $i \in I$. In terms of the hearts $\mathcal{A}_{h}$, the reflection functors amount to tilting at the simple objects $E_{i}^{h}$ in the sense of Happel-Reiten-Smal\o  \cite{hrs}.

Similarly, under the inverse equivalence $R\Psi_{h}^{-1}$, the standard $t$-structure on $D^{b}(\Pi_{\Gamma})$ gives a non-standard $t$-structure on $D^{b}_{\widetilde{G}}(\T)$. Restricting this $t$-structure to the subcategory $\D \subset D^{b}_{\widetilde{G}}(\T)$ of objects supported along the zero section gives a heart $\mathcal{B}_{h} \subset \D$ whose simple objects are the spherical objects $\E_{i}^{h}$ that we have already encountered. Proposition~\ref{twisttilt} shows how the action of the spherical twists can be described
in terms of tilting at the simple objects $\E_{i}^{h}$. 

Note that, although the spherical twists are indeed the right analogues of the reflection functors, the 
situation for $\T$ is richer than for $\mathbb{P}^{1}$, since the spherical twists actually act by {\it autoequivalences}
on the category $\D$, while the reflection functors are derived equivalences between categories of {\it different} quivers and the effect of the reflection functors on $D^{b}_{\widetilde{G}}(\mathbb{P}^{1})$ is merely to tilt $t$-structures.

Completing the analogy between $\mathbb{P}^{1}$ and $\T$, let us note that there is an isomorphism $K_{0}(D^{b}_{\widetilde{G}}(\mathbb{P}^{1})) \simeq K_{0}(\D)$ sending the class of $E_{i}^{h}$ to the class of $\E_{i}^{h}$, that these collections form bases of simple roots
for $K_{0}(D^{b}_{\widetilde{G}}(\mathbb{P}^{1})) \simeq K_{0}(\D)$ thought of as the affine root lattice, and that the spherical twists generate an action of a braid group $B_{\Gamma}$ on $\D$, which induces the action of the affine Weyl group on $K_{0}(D^{b}_{\widetilde{G}}(\mathbb{P}^{1})) \simeq K_{0}(\D)$ agreeing with that coming from the reflection functors.

\medskip
We summarize the relation between $\mathbb{P}^{1}$ and $\T$ in the following table.
\medskip
\begin{center}
\begin{tabular}{c|c}
$\mathbb{P}^{1}$ & $\T$ \\ 
\hline \hline \\
$R\Phi_{h}: D^{b}_{\widetilde{G}}(\mathbb{P}^{1}) \simeq D^{b}(Q_{h})$ & $R\Psi_{h}: D^{b}_{\widetilde{G}}(\T) \simeq D^{b}(\Pi_{\Gamma})$ \\\\\hline \\
Hearts $\mathcal{A}_{h} \subset  D^{b}_{\widetilde{G}}(\mathbb{P}^{1}) $ & Hearts $\mathcal{B}_{h} \subset \D \subset D^{b}_{\widetilde{G}}(\T)$\\\\
\hline \\
Simples $E_{i}^{h}$ & Spherical objects $\E_{i}^{h}$ \\\\
\hline \\ 
Reflection functors & Spherical twists\\\\
\hline \\
Affine Weyl group & Braid group \\\\
\hline
\end{tabular}
\end{center}

\medskip

Furthermore, together with the equivalences $\Pi_{\Gamma}\mbox{-mod} \simeq \Coh_{G}(\mathbb{C}^{2})$ and $D^{b}_{G}(\mathbb{C}^{2}) \simeq D^{b}(\widehat{X})$ for the resolution $\widehat{X} \rightarrow  \mathbb{C}^{2}/\!\!/G$, our results provide a chain of equivalences
$$D^{b}_{\widetilde{G}}(\T) \simeq D^{b}(\Pi_{\Gamma}) \simeq D^{b}_{G}(\mathbb{C}^{2}) \simeq D^{b}(\widehat{X}).$$
We may thus view $\T$ as providing a bridge between the McKay correspondence for $\mathbb{P}^{1}$ of \cite{kir} and the usual McKay correspondence for $\mathbb{C}^{2}$. 

\begin{rem}
Let us point out that the structures appearing in the above table are very similar to those in Bridgeland's paper \cite{tstruc}, in which
exceptional collections on certain Fano varieties are related to collections of spherical objects on canonical bundles. 
Since the combinatorics of affine $ADE$ diagrams with varying orientation is more complicated than that of the one-way oriented $A$ diagrams appearing in theory of exceptional collections, finding analogies for all of the fine results in \cite{tstruc} requires further work.
\end{rem}

In order to keep our exposition self-contained, we recall in Section~\ref{affinecase} some algebraic aspects of the McKay correspondence for
$\mathbb{C}^{2}$. In fact, in Theorem~\ref{equiv}, we give a geometric construction
of an algebra $B$ and an equivalence $\Coh_{G}(V) \rightarrow B\mbox{-mod}$ for any finite group $G$ acting linearly on a vector space $V$. We then show in Theorem~\ref{koszul} that our algebra $B$ is Koszul and use this to realize $B$ as quotient of the path algebra of the McKay quiver by quadratic relations. When $G \subset SL_{2}(\mathbb{C})$, we note that $B \simeq \Pi_{\Gamma}$, the preprojective algebra of the affine Dynkin diagram $\Gamma$ associated to $G$. Finally, in Lemma~\ref{lemmagammaconfigs} we characterize $\Pi_{\Gamma}$ as the Koszul dual of the $\Ext$-algebra of a collection of spherical objects. These results will be used in later sections
and should be of independent interest.


\section{The affine McKay correspondence and Koszul duality}\label{affinecase}



\medskip

Let $G$ be a finite group and $V$ a finite dimensional representation over $\mathbb{C}$. Let  $I$ be an index set for the irreducible representations of $G$, and for $i \in I$, let $W_{i}$ be the corresponding irreducible representation.
We think of $V$  as the total space of a $G$-equivariant vector bundle over a point, with projection $\pi$ and zero-section $s$:
\begin{equation*}
\xymatrix{
V \ar[d]_{\pi}\\
\bull \ar@/^{2pc}/[u]^{\quad}_{s}
}
\end{equation*}
Taking the pull-back $ \pi^{*}W$ of
the equivariant vector bundle $W=\oplus_{i}W_{i}$ on the point, we set 

\begin{equation}\label{thealgebra}
B:=\End_{G}( \pi^{*}W)^{\rm op}\simeq (\pi^{*}\End(W)^{\rm op})^{G}\simeq (S^{\bull}V^{*}\otimes \End(W)^{\rm op})^{G}.
\end{equation}

The following theorem shows that the algebra $B$ encodes everything there is to know about $G$-equivariant coherent sheaves on the affine space $V$.
\begin{theorem}\label{equiv}
There is an equivalence
 $$\Psi=\Hom_{G}(\pi^{*}W, -):\Coh_{G}(V) \longrightarrow B \mbox{{\rm -mod}}$$
from $G$-equivariant coherent sheaves to finitely generated left $B$-modules, where the action on an object $\Psi(\F)$ is given by precomposition
with elements of $B^{\rm op}=\End_{G}( \pi^{*}W)$.
\end{theorem}

\begin{proof}

We use the standard fact that if $\mathcal{A}$ is a cocomplete abelian category 
then the functor $\Hom(P,-): \mathcal{A} \rightarrow \End(P)^{\rm op} \mbox{{\rm -Mod}}$
is an equivalence if and only if $P$ is {\it small} ($\Hom(P,-)$ commutes with sums), {\it projective}, and a {\it generator} of $\mathcal{A}$ (given an object $X \in \mathcal{A}$ there is a surjection $P^{\oplus I} \twoheadrightarrow X$ for some index set $I$). See for instance \cite{morita}.

In our case, we take ${\rm QCoh}_{G}(V)$, the category of $G$-equivariant quasi-coherent sheaves on $V$, as our cocomplete abelian category.
Then the object $\pi^{*}W$ is small (the underlying quasi-coherent sheaf is free of finite rank and taking invariants commutes with sums) and projective (again because the sheaf is free of finite rank and taking invariants is exact).
To see that $\pi^{*}W$ is a generator, note that given any quasi-coherent $G$-sheaf $\mathcal{F}$, we can push-forward to the point to get a $G$-representation $\pi_{*}\mathcal{F}$, which splits as a sum of irreducibles. 
There is thus a surjection $W^{\oplus I} \twoheadrightarrow \pi_{*}\mathcal{F} $, which pulls-back to a surjection $\pi^{*}W^{\oplus I} \twoheadrightarrow \pi^{*}\pi_{*}\mathcal{F}$. Composing with the counit $\pi^{*}\pi_{*}\mathcal{F}  \rightarrow \mathcal{F}$, which is surjective since $\mathcal{F}$
is generated by its global sections, gives a surjection $\pi^{*}W^{\oplus I} \twoheadrightarrow \mathcal{F}$, showing that $\pi^{*}W$ is a generator 
for ${\rm QCoh}_{G}(V)$.  Thus the functor $\Psi=\Hom_{G}(\pi^{*}W, -):{\rm QCoh}_{G}(V) \rightarrow B \mbox{{\rm -Mod}}$ is an equivalence, by the above quoted fact.

Finally, we check that this equivalence restricts to an equivalence between coherent $G$-sheaves $\Coh_{G}(V)$ and finitely generated $B$-modules $B \mbox{{\rm -mod}}$. First note that we can realize every coherent $G$-sheaf $\mathcal{F}$ as
a quotient $\pi^{*}W^{\oplus n} \twoheadrightarrow \mathcal{F}$. We can do this
by first constructing a non-equivariant surjection $\mathcal{O}^{\oplus m} \twoheadrightarrow \mathcal{F}$, inducing to an equivariant surjection $\pi^{*}\mathbb{C}G^{\oplus m} \twoheadrightarrow \bigoplus_{g} \mathcal{F}$, and then
composing a surjection $\pi^{*}W^{\oplus n} \twoheadrightarrow
\pi^{*}\mathbb{C}G^{\oplus m}$ with a surjection $\bigoplus_{g} \mathcal{F} \twoheadrightarrow \mathcal{F}$. Applying the functor $\Psi$ to the surjection $\pi^{*}W^{\oplus n} \twoheadrightarrow \mathcal{F}$ gives a surjection
$B^{\oplus n} \twoheadrightarrow \Psi(\mathcal{F})$, so
$\Psi$ restricts to a functor from $\Coh_{G}(V)$ to 
$B \mbox{{\rm -mod}}$, which we already know to be full and faithful. To see that the restriction of $\Psi$ is essentially surjective, realize a finitely generated $B$-module $M$ as a quotient $B^{\oplus n} \twoheadrightarrow M$ and apply the inverse equivalence $\Psi^{-1}$ to get a surjection $\pi^{*}W^{\oplus n} \twoheadrightarrow \Psi^{-1}(M)$. Then the $G$-sheaf $\mathcal{F}=\Psi^{-1}(M)$ is coherent and $\Psi(\mathcal{F}) \simeq M$,
as desired.

\end{proof}

To understand the algebra $B$, we use some basic facts about Koszul algebras. \medskip


Let $B$ be a graded algebra with semisimple degree zero part $B_{0}$, which we also consider as a $B$-module via $B/B_{\geq 1} \simeq B_{0}$.

\begin{defn} $B$ is called {\it Koszul} if the algebra $$E(B):=\Ext^{\bull}_{B}(B_{0},B_{0})$$
is generated in degree $1$. 
\end{defn}

We assume further that $B$ is {\it  finite}, meaning that each $B_{i}$ is finitely generated as a left and a right $B_{0}$-module, and that $B$ is Noetherian. We summarize the facts that we need in the following theorem, the proofs of which can be found in \cite{bgs}.

\begin{theorem}\label{facts}

\noindent 1.  If $B$ is Koszul, then $B$ is {\it quadratic}, meaning that the natural ring homomorphism $T^{\bull}_{B_{0}}B_{1} \rightarrow B$
is a surjection with kernel generated by a space of relations $R$ in degree $2$. Define the quadratic dual $B^{!}$ to be the algebra with dual generators $B_{1}^{*}$ and dual relations $R^{\perp}$.

Note that the tensor algebra is taken over $B_{0}$.

\medskip

\noindent 2. If $B$ is Koszul, then so are $B^{!}$ and $E(B)$. There are canonical isomorphisms $E(B) \simeq {B^{!}}^{\rm op}$ and  $E(E(B)) \simeq {{({B^{!}}^{\rm op})}^{!}}^{\rm op} \simeq B$.

\medskip

\noindent 3. If $B$ is Koszul, it has finite global dimension.

\medskip

\noindent 4. (Numerical criterion) Assume $B$ is an algebra over a field $F$ and that there is a finite set of orthogonal idempotents $e_{i} \in B_{0}$ such that $B_{0}=\oplus_{i} Fe_{i}$. We can thus form a matrix of Poincar\'e series $$P(B,t)_{i,j}=\sum_{d} t^{d}{\rm dim}_{F} \; e_{i}B_{d}e_{j}.$$  
Since $E(B)_{0}= \Hom_{B_{0}}(B_{0},B_{0}) \simeq B_{0}^{\rm op} \simeq B_{0}$, we can also form the matrix $P(E(B),t)$. 

\medskip
Then $B$ is Koszul if and only if
$$P(B,t)\cdot P(E(B),-t)={\rm Id}.$$
\end{theorem}

\medskip \medskip
Let us return to the algebra $B=\End_{G}(\pi^{*}W) \simeq (S^{\bull}V^{*} \otimes \End(W)^{\rm op})^{G}$.

First, I claim that the algebra $B$ is Noetherian. To see this, first note that the endomorphism algebra $\End_{X}(\F)$ of any coherent sheaf $\F$ on a Noetherian scheme $X$ must be Noetherian, since any ascending chain of ideals $I_{1} \subseteq I_{2} \subseteq \cdots$ of $\End_{X}(\F)$ is in particular an ascending chain of submodules of $\End_{X}(\F)$ thought of as a finitely generated module over the Noetherian algebra $A=\mathcal{O}_{X}(X)$ and so must eventually terminate. If $X$ carries the action of a finite group, then by an ancient theorem of Noether $A$ is finite as a module over the ring of invariants $A^{G}$, and so the invariant endomorphisms $\End_{G}(\F)$ are a finitely generated module over $A^{G}$.
Then the same argument on the ascending chain of ideals $I_{1} \subseteq I_{2} \subseteq \cdots$ of $\End_{X}(\F)$ works to show that the algebra 
$\End_{G}(\F)$ is Noetherian.

Next, note that $B_{0}=\oplus_{i}\Homg(W_{i},W_{i}) \simeq \oplus_{i}\mathbb{C}\cdot 1_{W_{i}}$, a commutative semi-simple algebra. Given this property of $B_{0}$ and the fact that $B$ is Noetherian, we may try to apply the numerical criterion to check that our algebra $B$ is Koszul.

To apply the numerical criterion, we need to understand $E(B)$. By the adjunction $\pi^{*} \dashv \pi_{*}$, we see that the image of $s_{*}W$ under the equivalence $\Psi: D^{b}_{G}(V) \simeq {\rm perf}\; B$ is
$$\Hom_{G}(\pi^{*}W, s_{*}W) \simeq \Hom_{G}(W,\pi_{*} s_{*}W) \simeq \oplus_{i} \Hom_{G}(W_{i},W_{i}) \simeq B_{0}$$
and thus we have an isomorphism $E(B) = \Ext^{\bull}_{B}(B_{0},B_{0}) \simeq \Ext_{G}^{\bull}(s_{*}W,s_{*}W)$. The Koszul resolution $0 \rightarrow \pi^{*}\bigwedge^{n}V^{*} \otimes W \rightarrow 
\cdots \rightarrow  \pi^{*}V^{*} \otimes W \rightarrow \pi^{*}W \rightarrow s_{*}W \rightarrow 0$
and the adjunction $Ls^{*} \vdash s_{*}$ then give isomorphisms

\begin{equation}\label{extalg}
E(B) \simeq \Ext_{G}^{\bull}(s_{*}W,s_{*}W) \simeq (\bigwedge^{_{\;\;\bull}}\! V \otimes \End(W))^{G}.
\end{equation}

In order to apply the numerical criterion to our algebra $B$ we need expressions for the Poincar\'e series of the graded vector spaces $(S^{\bull} V^{*} \otimes U)^{G}$ and $(\bigwedge^{\!\! \bull}\! V \otimes U)^{G}$ where $U$ is some $G$-representation. Letting $S_{U}(t)$ be Poincar\'e series of the first and $E_{U}(t)$ the series of the second, Molien's formulae are

\begin{equation}\label{molien}
\begin{split}
S_{U}(t)=\frac{1}{|G|}\sum_{g \in G} \frac{\chi_{_{U}}(g)}{{\rm det}_{V}(1-g^{-1}\cdot t)} \\ \\
E_{U}(t)=\frac{1}{|G|}\sum_{g \in G} {\chi_{_{U}}(g)}{{\rm det}_{V}(1+g\cdot t)}
\end{split}
\end{equation}
A proof of the first formula can be found in \cite[Theorem 3.2.2]{invt}, and the second formula follows similarly.
\begin{theorem}\label{koszul}
The algebra $B=(S^{\bull}V^{*}\otimes \End(W)^{\rm op})^{G}$ is Koszul.
\end{theorem}

\begin{proof}
We use the numerical criterion and the isomorphism $E(B) \simeq (\bigwedge^{\!\! \bull}\! V \otimes \End(W))^{G}$.
The degree zero part of $B$ is commutative semisimple with one idempotent
$e_{i}=1_{W_{i}} \in \Hom_{G}(W_{i},W_{i})$ for each irreducible $W_{i}$. Thus we have matrices $S(t):=P(B,t)$ and $ E(t):=P(E(B),t)$ of Poincar\'e series with rows and columns indexed by irreducibles. 

We need to check that $S(t)\cdot E(-t)={\rm Id}$. The $(p,r)$ entry of the product takes the form $\sum_{q}S_{pq}(t)\cdot E_{qr}(-t)$. Letting $\chi_{kl}$ be the character of the representation $\Hom(W_{k},W_{l})$ and setting $\Delta_{g}={\rm det}_{V}(1-g^{-1}\cdot t)$, Molien's formulae give
$$S_{pq}=\frac{1}{|G|}\sum_{g}\frac{\chi_{qp}(g)}{\Delta_{g}} \;\;\;\; \mbox{and} \;\;\;\; E_{qr}(-t)=\frac{1}{|G|}\sum_{h}\chi_{qr}(h^{-1})\Delta_{h},$$
where in the expression for $S_{pq}(t)$ we have $\chi_{qp}$ instead of $\chi_{pq}$ because we have taken the opposite algebra
and in the expression for $E_{qr}(-t)$ we take the value of the character on $h^{-1}$ so that we have $\Delta_{h}$ instead of $\Delta_{h^{-1}}$.

Letting $\Delta=\prod_{g \in G} \Delta_{g}$, the $(p,r)$ entry of our product is
$$\sum_{q}S_{pq}(t)\cdot E_{qr}(-t)=\frac{1}{|G|^{2}\Delta}\sum_{q}\sum_{g,h}\chi_{q}(g^{-1})\chi_{p}(g)\chi_{q}(h)\chi_{r}(h^{-1})\Delta_{h}\prod_{k\neq g}\Delta_{k}.$$
We claim that this is equal to zero when $p \neq r$ and equal to one when $p=r$. Thus the product of our matrices will indeed be the identity. To see this, notice that
the contribution to the above expression coming from a fixed pair $g,h$ and summing over $q$ will be
$$\frac{1}{|G|^{2}\Delta}\chi_{p}(g)\chi_{r}(h^{-1})\Delta_{h}\prod_{k\neq g}\Delta_{k}\sum_{q}\chi_{q}(g^{-1})\chi_{q}(h).$$
The factor before the sum is constant for a fixed pair $g,h$ and the sum itself can be determined by the second orthogonality relation for characters: it is equal to zero when 
$g$ and $h$ are not conjugate and is equal to $|C_{G}(h)|$ when they are conjugate.

Let $g \sim h$ denote that $g$ and $h$ are conjugate in $G$. After summing up $q$ we are left with
\begin{gather*}
\sum_{q}S_{pq}(t)\cdot E_{qr}(-t)=\frac{1}{|G|\Delta}\sum_{\substack{ (g,h) \\ g \sim h}}\frac{|C_{G}(h)|}{|G|}\chi_{p}(g)\chi_{r}(h^{-1})\Delta_{h}\prod_{k\neq g}\Delta_{k}.
\end{gather*}
Note that the sum on the right is over conjugate pairs. 

For each $g$, we have $\Delta_{h}\prod_{k\neq g}\Delta_{k}=\Delta$, so we can cancel the $1/\Delta$ in front of the sum. The summand $\chi_{p}(g)\chi_{r}(h^{-1})$ only depends on the conjugacy class of $h$ since $g$ and $h$ are conjugate, and so the number of times $\chi_{p}(g)\chi_{r}(h^{-1})$ is counted in the sum is the number of elements in the conjugacy class of $h$. Since the factor $\frac{|C_{G}(h)|}{|G|}$ is precisely the reciprocal of this number,
we are left with 
$$\sum_{q}S_{pq}(t)\cdot E_{qr}(-t)=\frac{1}{|G|}\sum_{h} \chi_{p}(h)\chi_{r}(h^{-1})=\delta_{pr},$$ 
where the last equality is from orthogonality of the irreducible characters.
\end{proof}

\medskip

One interesting consequence of Koszulity for the algebra $B$ is that it must be quadratic, by Theorem~\ref{facts}. That is, the natural homomorphism $T^{\bull}_{B_{0}}B_{1} \rightarrow B$
is surjective with kernel generated in degree $2$. In fact, since $B_{0}=\oplus_{i \in I}\Hom_{G}(W_{i}, W_{i})$ is commutative semisimple with primitive orthogonal idempotents $e_{i}=1 \in \Hom_{G}(W_{i}, W_{i})$, we can realize $T^{\bull}_{B_{0}}B$ as the path algebra of a quiver whose vertices are labeled by the $e_{i}$ and whose arrows from $e_{i}$ to $e_{j}$ are identified with a basis for $e_{i}B_{1}e_{j}$. Since in fact $e_{i}B_{1}e_{j}=(V^{*} \otimes \Hom(W_{i}, W_{j}))^{G} \simeq \Hom_{G}(W_{i}, V \otimes W_{j})$,
we get precisely the McKay quiver, as usually defined. ( Let us point out here that the definition of the McKay quiver is therefore not as fanciful as might first appear. ) Thus $B$ itself is realized as the quotient of the path algebra of the McKay quiver by some quadratic relations.

In general the relations can be quite difficult to write down. The best method here is to find a `superpotential' for the algebra $B$ whose derivatives give the relations. For more on this and for another approach to proving Koszulity for an algebra isomorphic to $B$, see the paper of Bocklandt-Schedler-Wemyss \cite{super}.

\subsection{Spherical objects}\label{sectsphericalobjects}

For our purposes, having a presentation of $B$ as the quotient of a path algebra will be less important than the characterization of $B$ in terms of configurations of spherical objects, as introduced by Seidel and Thomas \cite{st}. See \cite[Chapter 8]{huy} for a nice exposition.

In our applications we'll be interested in a smooth, quasi-projective surface $X$ carrying the action of a finite group $G$ and whose canonical bundle is trivial as a $G$-sheaf (for instance, $G \subset SL_{2}(\mathbb{C})$, $X=V$ a $2$-dimensional vector space). In this case, an object $\E \in D^{b}_{G}(X)$ is called {\it spherical} if there is a graded ring isomorphism $\Ext^{\bull}_{G}(\E,\E) \simeq H^{\bull}(S^{2},\mathbb{C})$, where $S^{2}$ is the $2$-sphere. 
 
Given a graph $\Gamma$, a $\Gamma$-{\it configuration} of spherical objects is a collection of spherical objects $\E_{i}$ indexed by the nodes of $\Gamma$ such that $${\rm dim}\;  \Ext^{1}(\E_{i},\E_{j})=\# \; \mbox{edges from $i$ to $j$}$$
and all other $\Ext$s vanish.

In all the cases that we consider, the objects $\E_{i}$ lie in a full triangulated subcategory $\D$ of $D^{b}_{G}(X)$ consisting of objects whose cohomology is supported on a fixed compact subvariety.  Compact support ensures that all $\Hom$s are finite dimensional and that the subcategory $\D$ has Serre duality, while triviality of $\omega_{X}$ means that duality takes the simple form $\Hom_{\D}(\F,\G) \simeq \Hom_{\D}(\G, \F[2])^{*}$. A triangulated category with these properties is commonly called  $2$-{\it Calabi-Yau} ($2$-{\it CY}).

\medskip

A much-studied example (see for instance \cite{klein}) is when $X=V$, a $2$-dimensional vector space, $G \subset SL(V)$, and $\D \subset D^{b}_{G}(V)$ is the full triangulated subcategory supported at the origin. Here we have a $\Gamma$-configuration for $\Gamma$ the affine Dynkin diagram associated to $G$. The spherical objects are $s_{*}W_{i} \simeq W_{i} \otimes {\mathcal O}_{0}$ for the irreducible $G$-representations $W_{i}$. To see that these objects are spherical, recall the isomorphism \ref{extalg}:
$$\Ext^{\bull}_{G}(s_{*}W_{i}, s_{*}W_{i}) \simeq (\bigwedge^{_{\;\;\bull}}\! V \otimes \Hom(W_{i},W_{i}))^{G}.$$ Since $V \otimes \Hom(W_{i},W_{i})$ has no invariants by McKay's observation and since $\bigwedge^{\! 2}V$ is trivial, the latter algebra is indeed isomorphic to $H^{\bull}(S^{2},\mathbb{C})$.
To see that the objects $s_{*}W_{i}$, $i \in \Gamma$ form a $\Gamma$-configuration, note that if $i \neq j$, then $\Ext_{G}^{\bull}( s_{*}W_{i}, s_{*}W_{j}) \simeq (\bigwedge^{\!\! \bull}\! V \otimes \Hom(W_{i},W_{j}))^{G}$ is zero
in degrees $0$ and $2$ by Schur's lemma and $${\rm dim} \; \Ext_{G}^{1}( s_{*}W_{i}, s_{*}W_{j}) ={\rm dim}\; \Hom_{G}(W_{i},V \otimes W_{j})=\# \; \mbox{edges beween $i$ and $j$}$$ by McKay's observation. 

\begin{rem}
It is known that the classes of the $\E_{i}$ form a basis for $K_{0}(\D)$ (compare with Proposition~\ref{groth}). The above discussion shows that in the basis  $\E_{i}$, the natural {\it Euler form} $$\langle \E,{\mathcal F}\rangle=\sum_{k}(-1)^{k}{\rm dim}\,\Extg^{k}(\E,\F)$$ on $K_{0}(\D)$ is given by the Cartan matrix of $\Gamma$, so we may identify $K_{0}(\D)$ with the affine root lattice associated to $\Gamma$ and the $\E_{i}$ with a base of simple roots.

In this way we think of $\D$ as a categorification of the affine root lattice, with bases of simple roots
replaced by $\Gamma$-configurations of spherical objects and the action of the Weyl group
replaced by the action of the braid group $B_{\Gamma}$. For more on this point of view, see Khovanov and Huerfano \cite{catadj}, who use the algebra $E(B)$ to categorify the adjoint representation of the quantum group associated to $\Gamma$.
\end{rem}

The following lemma shows that in dimension two this example is in some sense universal.


\begin{lemma}\label{lemmagammaconfigs}
\begin{enumerate}
\item The objects $s_{*}W_{i}$, $i \in \Gamma$ form a $\Gamma$-configuration of spherical objects in the subcategory $\D \subset D^{b}_{G}(\mathbb{C}^{2})$ of objects supported at the origin.

\item The $\Ext$-algebra of this $\Gamma$ configuration,

$$\Ext_{G}^{\bull}(\bigoplus_{i \in \Gamma} \E_{i},\bigoplus_{i \in \Gamma} \E_{i}),$$
is Koszul with Koszul dual $\Pi_{\Gamma}$, the preprojective algebra of the diagram $\Gamma$.

\item
Let $\Gamma$ be an affine Dynkin diagram of type {\rm ADE}, $\E_{i}$, $i \in \Gamma$ a $\Gamma$-configuration of spherical objects
in a $2$-CY category $\D$, and $\E^{\prime}_{i}$, $i \in \Gamma$ another $\Gamma$-configuration in a possibly different $2$-CY category $\D^{\prime}$. Then the $\Ext$-algebras of the two $\Gamma$-configurations are isomorphic:

$$\Ext^{\bull}_{\D}(\bigoplus_{i \in \Gamma} \E_{i}, \bigoplus_{i \in \Gamma} \E_{i}) \simeq \Ext^{\bull}_{\D^{\prime}}(\bigoplus_{i \in \Gamma} \E^{\prime}_{i}, \bigoplus_{i \in \Gamma} \E^{\prime}_{i}).$$

By 2., any such $\Ext$-algebra is Koszul.
\end{enumerate}
\end{lemma}

\begin{proof}
\noindent 1. We have already seen this in the discussion before the lemma.

\medskip
\noindent 2. This is well-known and not difficult. See Example~\ref{typeA} for an argument in Type A. Types D and E are dealt with similarly.

\medskip
\noindent 3. Let $\{\E_{i}\}$ be a $\Gamma$-configuration in a $2$-CY category. Then
$\Hom(\E_{i}, \E_{i}) \simeq \mathbb{C}$ and $\Ext^{2}(\E_{i}, \E_{i}) \simeq \Hom(\E_{i}, \E_{i})^{*} \simeq \mathbb{C}$ by sphericity and Serre duality and all other $\Hom$s and $\Ext^{2}$s are zero by the condition of being a $\Gamma$-configuration. Thus the composition $\Ext^{1}(\E_{j}, \E_{k}) \otimes \Ext^{1}(\E_{i}, \E_{j}) \rightarrow \Ext^{2}(\E_{i}, \E_{k})$ is zero unless $i=k$, in which case the composition $$\Ext^{1}(\E_{j}, \E_{i}) \otimes \Ext^{1}(\E_{i}, \E_{j}) \rightarrow \Ext^{2}(\E_{i}, \E_{i}) \simeq \mathbb{C}$$
is just the Serre pairing and we have
$\Ext^{1}(\E_{i}, \E_{j}) \simeq  \Ext^{1}(\E_{j}, \E_{i})^{*}$.

Now let $\{\E_{i}^{\prime}\}$ be another $\Gamma$-configuration in a possibly different $2$-Calabi-Yau category.
To establish an isomorphism between the $\Ext$-algebras of $\oplus_{i}\E_{i}$ and $\oplus_{i} \E_{i}^{\prime}$, it is enough
to take the natural identifications for $\Hom$s and $\Ext^{2}$s and then to give an isomorphism on $\Ext^{1}$s compatible with the above pairing. To achieve this, choose for every pair of adjacent vertices $i$ and $j$ a postive direction $i \rightarrow j$ and give an isomorphism $\Ext^{1}(\E_{i}, \E_{j})\simeq \Ext^{1}(\E_{i}^{\prime}, \E_{j}^{\prime})$. Then letting the isomorphism for the negative direction $j \rightarrow i$ be determined by duality ensures compatability with the pairing.
\end{proof}

\begin{rem}\label{useful}
The lemma is very useful. If $B$ is any graded algebra for which $E(B)$ is an $\Ext$-algebra of a $\Gamma$-configuration, then by the third part of the lemma there is an isomorphism
$$E(B) \simeq \Ext_{G}^{\bull}(\bigoplus_{i \in \Gamma} \E_{i},\bigoplus_{i \in \Gamma} \E_{i}).$$ Then by Koszul duality (Theorem~\ref{facts}) and the second part of the lemma,
$$B \simeq E(E(B))\simeq \Pi_{\Gamma}.$$

This will be important in Section~\ref{sectreflectfunctors}, where we use $\Gamma$-configurations of spherical objects to relate equivariant sheaves on the cotangent bundle $\T$ to the above universal example.
\end{rem}

\begin{example}\label{typeA} Let $G=\mathbb{Z}/n\mathbb{Z}$. Defining $W_{1}$ to be the irreducible representation of $\mathbb{Z}/n\mathbb{Z}$ where $1$ acts as multiplication by $\zeta=e^{2\pi\, i/n}$, all of the other irreducible representations are just powers of this one, which we denote by $W_{0}, W_{1}, \dots, W_{n-1}$. The group $\mathbb{Z}/n\mathbb{Z}$ embeds in $SL_{2}(\mathbb{C})$ with cyclic generator

\begin{displaymath}
\left( \begin{array}{ccc} 
\zeta & 0\\ 
0 & \zeta^{-1}
\end{array} \right)
\end{displaymath}
so that the standard  representation takes the form $V=W_{1} \oplus W_{n-1}$. In this case it is straight-forward to check that the McKay quiver is 

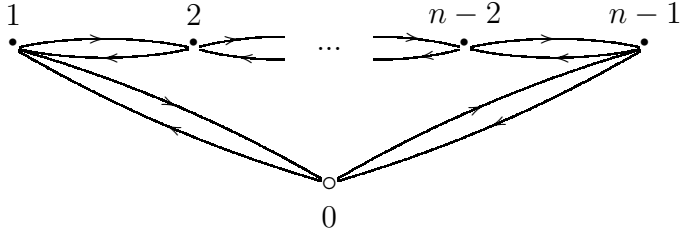
\begin{equation*}
\xygraph{ !~:{@{-}|@{>}}
!{<0cm,0cm>;<.6cm,0cm>:<0cm,.6cm>::}
!{(7,-3.75)}*{0}
!{(0,.75)}*{1}
!{(4,.75)}*{2}
!{(10,.75)}*{n-2}
!{(14,.75)}*{n-1}
!{(0,0)}*{\bull}="a"
!{(4,0)}*{\bull}="b"
!{(7,0)}*{...}
!{(6,.25)}*{}="d"
!{(6,-.25)}*{}="D"
!{(8,.25)}*{}="h"
!{(8,-.25)}*{}="H"
!{(10,0)}*{\bull}="e"
!{(14,0)}*{\bull}="f"
!{(7,-3)}*{\circ}="g"
"a":@/^.3pc/"b", "b":@/^.3pc/"a" , "b":@/^.1pc/"d","D":@/^.1pc/"b", "e":@/^.3pc/"f",  "f":@/^.3pc/"e", "a":@/^.35pc/"g", "g":@/^.35pc/"a", "f":@/^.35pc/"g", "g":@/^.35pc/"f", "h":@/^.1pc/"e","e":@/^.1pc/"H", 
}
\end{equation*}

To find a presentation for $B=\End_{G}(\pi^{*}W)^{\rm op}$ we first find a presentation for $E(B)=\Ext^{\bull}_{G}(s_{*}W, s_{*}W)$ and then compute its quadratic dual $E(B)^{!}$. By Theorem~\ref{facts}, there is a canonical isomorphism $B \simeq {E(B)^{!}}^{\rm op}$, so this will be sufficient. 

First we find a nice basis of generators for $E(B)_{1}$ over $E(B)_{0}=\bigoplus_{i}\Hom_{G}(s_{*}W_{i},s_{*}W_{i}) \simeq \bigoplus_{i}\Hom_{G}(W_{i},W_{i}) \simeq B_{0}$.
Each summand  in $E(B)_{1}= \bigoplus_{i \rightarrow j}\Ext^{1}_{G}(s_{*}W_{i},s_{*}W_{j})$
is one dimensional since the $s_{*}W_{i}$ form a $\Gamma$-configuration. Choose non-zero clockwise arrows
$(i| i+1) \in \Ext^{1}(s_{*}W_{i},s_{*}W_{i+1})$ and define the counterclockwise arrows 
$(i+1| i) \in \Ext^{1}(s_{*}W_{i},s_{*}W_{i-1})$ by requiring that the whole collection
form Darboux coordinates for the antisymmetric Serre pairing

\begin{equation*}
\xymatrix{
\Ext^{1}_{G}(s_{*}W, s_{*}W) \otimes \Ext^{1}_{G}(s_{*}W, s_{*}W) \ar[r] & \Ext^{2}_{G}(s_{*}W, s_{*}W) \ar[r]^{\rm \;\;\;\;\;\;\;\;\;\;\;\; \;\;Tr} & \mathbb{C}
}
\end{equation*}
Denoting the product $(j|k)\cdot (i|j)=(i|j|k)$, then component by component, this means that ${\rm Tr}((i|i+1| i))=1$ and
${\rm Tr}((i+1|i|i+1))=-1$.

Since $\Ext^{2}(s_{*}W_{i},s_{*}W_{k}) = 0$ for $k \neq i$, we certainly have relations $$(j|k)\otimes (i| j)=0$$ 
whenever $k \neq i$. By the antisymmetry of the Serre pairing we also have a relation $$(i+1|i)\otimes (i|i+1) + (i-1|i)\otimes (i|i-1)=0$$
for each $i$. I claim that these two kinds of relations form a basis for the space $R$ of relations, that is, for the kernel of
$\Ext^{1}_{G}(s_{*}W, s_{*}W) \otimes_{B_{0}} \Ext^{1}_{G}(s_{*}W, s_{*}W) \rightarrow \Ext^{2}_{G}(s_{*}W, s_{*}W)$. Indeed, the space on the left has dimension equal to the number of paths of length two, while the space on the right has dimension equal to the number of nodes of the quiver.
Since the map is surjective, the kernel has dimension equal to the number of length two paths minus the number of nodes. But this is precisely the number of relations that we have given. It will thus be enough to see that our relations span the kernel. So consider a relation on paths of length two: $\sum a_{ijk}(i|j|k)=0$. Since $(i|j|k)=0$ when $k \neq i$, we can ignore those terms, and by the splitting $\Ext^{2}_{G}(s_{*}W,s_{*}W) = \bigoplus_{i}\Ext^{2}_{G}(s_{*}W_{i},W_{i})$, we can
consider the individual pieces $a_{ii+1i}(i|i+1|i) + a_{i i-1 i}(i|i-1|i)=0$ of the sum. Taking the trace of this latter relation, we see that $a_{ii+1i}-a_{i i-1 i}=0$, so the relation
is a scalar multiple of $(i|i+1|i) + (i|i-1|i)=0$, and so our relations are indeed sufficient.

\medskip

Having found bases of generators and of relations for $E(B)$ we can compute the quadratic dual $E(B)^{!}$. Let $\langle i| i+1\rangle$ and $\langle i| i-1\rangle$ be dual to $(i+1|i)$ and $(i-1|i)$. I claim that a basis for the dual relations $R^{\perp}$ is given by the elements

\begin{equation}\label{preproj}
\langle i+1|i \rangle\otimes \langle i| i+1\rangle - \langle i-1|i\rangle \otimes \langle i| i-1\rangle
\end{equation}
for each $i$. The number of such elements is equal to the number of nodes, which is exactly the dimension of $R^{\perp}$, and they certainly kill the space $R$, so it is enough to see that they span $R^{\perp}$. Suppose that $\sum b_{ijk}  \langle j | k \rangle  \otimes \langle i | j \rangle \in R^{\perp}$. Then in particular it must kill $(m|n)\otimes (l|m)$ when $l \neq n$, and so $b_{nml}=0$.
This leaves $\sum b_{iji}  \langle j | i \rangle  \otimes \langle i | j \rangle$, which must kill
$(k+1|k)\otimes (k|k+1) + (k-1|k)\otimes (k|k-1)$ for all $k$, so $b_{kk+1k}+b_{kk-1k}=0$, and we see our relations are sufficient.

If $\alpha=\langle i| i+1\rangle$, then $\overline{\alpha}=\langle i+1 | i\rangle$ and the relation
in \ref{preproj} is precisely the $i$th component of the preprojective relations.
\end{example}


\section{McKay correspondence for ${\mathbb P}^{1}$}\label{sectprojmckay}
We review here Kirillov \cite{kir} and prove some related facts that will be useful later.

Let $V$ be a 2-dimensional vector space, set $\mathbb{P}^{1}=\mathbb{P}(V)$, and assume that our finite subgroup $G \subset SL(V)$ contains $\pm I$.  We divide $G$-representations and $G$-sheaves into two types, {\it even} and {\it odd}, depending on whether $-I$ acts trivially or non-trivially.  

We shall be mostly interested in coherent, even $G$-sheaves, which we can also think of as $\widetilde{G}=G/\pm I$-sheaves, where $\widetilde{G}$ is now a subgroup of $PSL(V)$. We denote by ${\rm Coh}_{\widetilde{G}}$ the category whose objects are even $G$-sheaves and whose morphisms
lie in  $\Homg$, the invariant part of $\Hom$ in the category of coherent sheaves. ${\rm Coh}_{\widetilde{G}}$ is abelian and we denote its bounded derived category by $D^{b}_{\widetilde{G}}(\mathbb{P}^{1})$.

It will be convenient to work with odd sheaves as well. For instance, the $G$-action on the trivial bundle $V$ stabilizes the tautological sub-bundle $\mathcal{O}(-1)$. With this natural $G$-action, $\mathcal{O}(-1)$ is an odd sheaf since $-I$ acts non-trivially on the fibres.
As a tensor power of $\mathcal{O}(-1)$, the line bundle $\mathcal{O}(d)$ inherits a natural $G$-action, and its parity as a $G$-sheaf  agrees with the parity of the integer $d$. Thus $\mathcal{O}(d) \otimes W_{i}$ will be an even $G$-sheaf precisely when
the integer $d$ and the representation $W_{i}$ have the same parity. 

In order to keep track of even $G$-sheaves of the above form, we introduce a {\it parity function} $p: \Gamma \rightarrow \mathbb{Z}$ on the vertices of $\Gamma$, where $p(i)=0$ if the irreducible $G$-representation $W_{i}$ is even and $p(j)=1$ if the irreducible $W_{j}$ is odd. Notice also that if two edges in the diagram $\Gamma$ are connected, then they have opposite parity and their heights differ by one.   
 
Generalizing these properties of the parity function, we define a {\it height function} to be a function $h: \Gamma \rightarrow \mathbb{Z}$ on the vertices of $\Gamma$ satisfying the conditions:

\begin{enumerate}

\item $h(i) \equiv p(i)\;{\rm mod}\, 2$,

\item $|h(i) - h(j)|=1$ if $i$ is connected to $j$ in $\Gamma$. 

\end{enumerate}

The first condition says that the parity of the height of a vertex agrees with the parity of the representation $W_{i}$, so each height function $h$ gives rise to a collection of even $G$-sheaves
$$F_{i}^{h} =W_{i} \otimes {\mathcal O}(h(i))$$
indexed by the nodes $i \in \Gamma$. The second condition says that the height goes up or down one step between neighboring vertices of $\Gamma$. The height function then determines an orientation on the edges of $\Gamma$ by letting the edges flow downhill. We denote the resulting quiver by $Q_{h}$. 

\begin{example} Let $G=\mathbb{Z}/n\mathbb{Z}$ and take the height $h$ to be equal to the parity function $p$. The resulting quiver is pictured below, with the vertices labeled by the even $G$-sheaves.
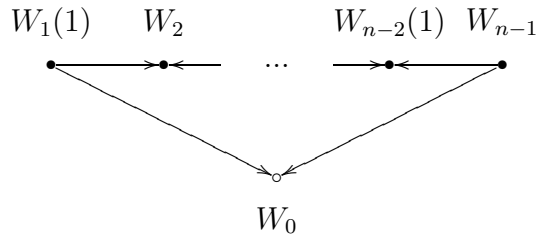
\begin{equation*}
\xygraph{
!{<0cm,0cm>;<.75cm,0cm>:<0cm,.75cm>::}
!{(4,-1.75)}*{W_{0}}
!{(0,1.75)}*{W_{1}(1)}
!{(2,1.75)}*{W_{2}}
!{(6,1.75)}*{W_{n-2}(1)}
!{(8,1.75)}*{W_{n-1}}
!{(0,1)}*{\scriptstyle\bullet}="a"
!{(2,1)}*{\scriptstyle\bullet}="b"
!{(3,1)}*{}="c"
!{(4,1)}*{...}
!{(5,1)}*{}="d"
!{(6,1)}*{\scriptstyle\bullet}="e"
!{(8,1)}*{\scriptstyle\bullet}="f"
!{(4,-1)}*{\scriptstyle\circ}="g"
"a":"b", "c":"b", "d":"e", "f":"e", "a":"g", "f":"g"
}
\end{equation*}
\end{example}

We can now give a main result of \cite{kir}, stated in a form convenient for us.


\begin{theorem}\label{kir}
Let $A_{h}={\rm End}_{G}(\oplus_{i} F_{i}^{h})^{\rm op}$ be the opposite algebra of the endomorphism algebra of the collection $F_{i}^{h}$, $i \in \Gamma$. 
Then the natural functor
$R\Phi_{h}:= R\Hom_{G}(\oplus_{i} F_{i}^{h},-)$

\begin{equation*} 
\xymatrix{
D^{b}_{\widetilde{G}}(\mathbb{P}^1) \ar[r]^{R\Phi_{h}} & D^{b}(A_{h})  
}
\end{equation*}
to the bounded derived category of finitely generated left $A_{h}$-modules is an equivalence. 
\end{theorem}

Kirillov shows moreover that there is an isomorphism of algebras $A_{h} \simeq \mathbb{C}Q_{h}$, where $\mathbb{C}Q_{h}$ is the path algebra of the quiver $Q_{h}$, so the functor can also be thought of as taking values in the derived category of representations of $Q_{h}$. More precisely, under the isomorphism $A_{h} \simeq \mathbb{C}Q_{h}$, the space of paths in $Q_{h}$ from $i$ to $j$ is given by $\Hom_{G}(F_{j}^{h}, F_{i}^{h}) = e_{j} A_{h} e_{i}$, where the identity morphisms $e_{i}=1 \in \Hom_{G}(F_{i}^{h}, F_{i}^{h})$ form a set of primitive, orthogonal idempotents in $A_{h}$.

\medskip

In the rest of this section we give some definitions and comments about the category $D^{b}_{\widetilde{G}}(\mathbb{P}^1) $ that will be useful later. We let $T$ denote the tangent bundle of $\mathbb{P}^1$ and $\omega$ the canonical bundle. Since $G \subset SL_{2}(\mathbb{C})$, any  isomorphisms $T \simeq \mathcal{O}(2)$ and $\omega \simeq \mathcal{O}(-2)$ will be $G$-equivariant. 

\medskip

We call a vertex $i$ of $Q_{h}$ a {\it sink} if it has lower height than its neighbors, so that arrows are coming in to $i$, and a {\it source} if 
it has greater height than its neighbors, so that arrows are coming out of $i$. Given a height function $h$ for which $i$ is a sink or a source, define a new height function $\sigma_{i}^{+}h$ or $\sigma_{i}^{-}h$ by

\begin{displaymath} 
\sigma_{i}^{+}h(k)= \left\{ \begin{array}{ll} 
h(k) & \textrm{if $i \neq k$}\\ 
h(k) + 2  & \textrm{if $i =k$}
\end{array} \right. \;\;\;\;\;\;\; {\rm or} \;\;\;\;\;\;\;
\sigma_{i}^{-}h(k)= \left\{ \begin{array}{ll} 
h(k) & \textrm{if $i \neq k$}\\ 
h(j) - 2  & \textrm{if $i =k$}.
\end{array} \right.
\end{displaymath}

Since we have assumed $\pm I \subset G$, the Dynkin diagrams that we are considering are bipartite and one can check that one height function differs from another by a sequence of such operations, turning sinks into sources and sources into sinks.


\medskip

The following observation of Kirillov is essential. If $i \in Q_{h}$ is a source, then
$$V \otimes F_{i}^{h}(-1) \simeq V \otimes W_{i}\otimes {\mathcal O}(h(i)-1) \simeq \bigoplus_{i \rightarrow j}W_{j}(h(j)) =\bigoplus_{i \rightarrow j} F_{j}^{h},$$
where the first isomorphism is by definition of $F_{i}^{h}$, the second by McKay's observation and the step-wise nature of height functions, and the last is again by definition. Here the sum is over arrows $i \rightarrow j$ leaving the source $i$. Thus 
tensoring the Euler sequence $0 \rightarrow {\mathcal O}(-1) \rightarrow V \rightarrow T(1) \rightarrow 0$ with $F_{i}^{h}(-1)$ and using the above isomorphisms gives

\begin{equation}\label{euler}
0 \rightarrow F_{i}^{\sigma_{i}^{-}h} \rightarrow \bigoplus_{i \rightarrow k}F_{k}^{h} \rightarrow F_{i}^{h} \rightarrow 0.
\end{equation}
Likewise, if $i \in Q_{h}$ is a sink, then tensoring the Euler sequence with $F_{i}^{h}(1)$ gives

\begin{equation}\label{euler2}
0 \rightarrow F_{i}^{h} \rightarrow \bigoplus_{k \rightarrow i}F_{k}^{h} \rightarrow F_{i}^{\sigma_{i}^{+}h} \rightarrow 0.
\end{equation}


\begin{lemma}\label{exts}
$\Extg^{1}(F_{k}^{h}, F_{l}^{h} \otimes T^{\otimes d})=0$ for all height functions $h$, all $k,l \in I$, and all $d \in \mathbb{N}$.
\end{lemma}

\begin{proof}
First we check that the statement is true when $h$ is the parity function. Then we show that if the statement is true for a height function $h$,
it is also true for the modified height functions $\sigma_{i}^{+}h$ and $\sigma_{i}^{-}h$. Since every height can be obtained from the parity function by a sequence of such modifications, this will establish the lemma.

By Serre duality, $\Ext_{G}^{1}(F_{k}^{h}, F_{l}^{h} \otimes T^{\otimes d}) \simeq \Hom_{G}(F_{l}^{h}\otimes T^{\otimes (d+1)}, F_{k}^{h})^{*}$. 
If $h$ is the parity function, the latter space is zero since $F_{l}^{h} \otimes T^{\otimes (d+1)}$ has higher degree than $F_{k}^{h}$, so the lemma holds for the parity function.  

Now assume the lemma is true for a height $h$ and that $i$ is a sink in $h$. We want to see that the lemma must hold for $\sigma_{i}^{+}h$. Consider the possible values of $k$ and $l$.

If $k=l=i$, then $\Ext_{G}^{1}(F_{i}^{\sigma_{i}^{+}h}, F_{i}^{\sigma_{i}^{+}h} \otimes T^{\otimes d})\simeq \Ext_{G}^{1}(F_{i}^{h}, F_{i}^{h} \otimes T^{\otimes d})$, since tensoring with ${\mathcal O}(2)$ is an equivalence. The latter space is zero by assumption on $h$.

If $k,l \neq i$, then $\sigma_{i}^{+}h(k)=k$ and we have $\Ext_{G}^{1}(F_{k}^{\sigma_{i}^{+}h}, F_{l}^{\sigma_{i}^{+}h} \otimes T^{\otimes d})\simeq \Ext_{G}^{1}(F_{k}^{h}, F_{l}^{h} \otimes T^{\otimes d})$. The latter space is zero by assumption on $h$. 

If $l=i$ and $k \neq i$, then $\sigma_{i}^{+}h(k)=k$ and $F_{l}^{\sigma_{i}^{+}h} = F_{i}^{\sigma_{i}^{+}h} \simeq F_{i}^{h}(2)$, so $\Ext_{G}^{1}(F_{k}^{\sigma_{i}^{+}h}, F_{l}^{\sigma_{i}^{+}h} \otimes T^{\otimes d}) \simeq \Ext_{G}^{1}(F_{k}^{h}, F_{i}^{h} \otimes T^{\otimes(d + 1)})$. The latter space is zero by assumption.

Finally, consider the case $k=i$ and $l \neq i$. If $d \geq 1$, 
$\Ext_{G}^{1}(F_{i}^{h}(2), F_{l}^{h} \otimes T^{\otimes d}) \simeq \Ext_{G}^{1}(F_{i}^{h}, F_{l}^{h} \otimes T^{\otimes(d -1)})$, and we are done by assumption.
If $d=0$,  then  by Serre duality $\Ext_{G}^{1}(F_{i}^{h}(2), F_{l}^{h}) \simeq \Hom_{G}(F_{j}^{h}, F_{i}^{h})^{*}$. The dimension of the latter space is the number of paths from $i$ to $j$ in $Q_{h}$, which is zero since $i$ is a sink. 

Thus if the lemma holds for a height function $h$, then it also holds for $\sigma_{i}^{+}h$. A similar argument shows that it also holds for $\sigma_{i}^{-}h$.
\end{proof}


For the proof of the next lemma we use Beilinson's resolution of the diagonal \cite{beil}, which on $\mathbb{P}^1 \times \mathbb{P}^1$ takes the form
$$0 \rightarrow p^{*}{\mathcal O}(-1)\otimes q^{*}\omega(1) \rightarrow p^{*}{\mathcal O} \otimes q^{*}{\mathcal O} \rightarrow {\mathcal O}_{\Delta} \rightarrow 0,$$
for $p$ and $q$ the projections of  $\mathbb{P}^1 \times \mathbb{P}^1$ onto the left and right factors respectively.
The resolution is canonically constructed and so is automatically $G$-equivariant, and in fact each of its terms is an even $G$-sheaf.

Taking $\F \in D({\rm QCoh}_{\widetilde{G}}(\mathbb{P}^1))$ and using the resolution of the diagonal as the kernel of a derived integral transform, we get an exact triangle
$$ Rq_{*} (p^{*}\F(-1)\otimes q^{*}\omega(1)) \rightarrow Rq_{*}(p^{*}\F\otimes q^{*}{\mathcal O} ) \rightarrow Rq_{*}( p^{*}\F \otimes {\mathcal O}_{\Delta})$$
in $D({\rm QCoh}_{\widetilde{G}}(\mathbb{P}^1))$. Applying the projection formula and flat base-change then gives the exact triangle

\begin{equation}\label{triangle}
R\Gamma(\F(-1))\otimes \omega(1) \rightarrow R\Gamma(\F)\otimes {\mathcal O} \rightarrow \F.
\end{equation} 
Notice that since $R\Gamma(\F) \otimes \mathcal{O} \in D({\rm QCoh}_{\widetilde{G}}(\mathbb{P}^1))$, $R\Gamma(\F)$ must be a complex of even $G$-representations. Similarly, since $\omega(1)$ is an odd sheaf and $R\Gamma(\F(-1)) \otimes \omega(1) \in D({\rm QCoh}_{\widetilde{G}}(\mathbb{P}^1))$, $R\Gamma(\F(-1))$ must be a complex of odd representations in order to make the tensor product with $\omega(1)$ even.


\begin{lemma}\label{span}
For any height function $h$, the collection $F_{i}^{h}$ generates $D({\rm QCoh}_{\widetilde{G}}(\mathbb{P}^1))$ in the sense that  if $R\Hom_{G}(F_{i}^{h}, \F)=0$ for all $i$, then $\F=0$.
\end{lemma}

\begin{proof}
As in the proof of Lemma~\ref{exts}, we first show the statement is true when $h$ is the parity function and then show that if the statement is true for a height function $h$, then it is also true for $\sigma_{i}^{+}h$ and $\sigma_{i}^{-}h$.

First let $h$ be the parity function and assume $R\Hom_{G}(F_{i}^{h}, \F)=0$ for all $F_{i}^{h}$. Note that $R\Gamma(\F)=R\Hom({\mathcal O},\F)$ and $R\Gamma(\F(-1)) \simeq R\Hom({\mathcal O}(1),\F)$. We claim that the assumption $R\Hom_{G}(F_{i}^{h}, \F)=0$ implies both are zero, and so $\F=0$ from the exactness of triangle \ref{triangle}. 

Indeed, if $R\Gamma(\F)=R\Hom({\mathcal O},\F)$ were non-zero, then since it consists of even representations it would contain some non-zero irreducible even representation $W_{i}$ and we would have $(W_{i}^{*} \otimes R\Hom({\mathcal O},\F))^{G} \simeq R\Homg(W_{i},\F) \neq 0$. This contradicts the assumption that $R\Hom_{G}(F_{i}^{h}, \F)=0$ for all $i$, since $W_{i}=F_{i}^{h}$ for $h$ the parity function and $i$ even.
Likewise, by assumption $R\Hom_{G}(W_{i}(1),\F) \simeq (W_{i}^{*}\otimes R\Hom({\mathcal O}(1),\F))^{G}=0$ for all odd representations $W_{i}$, but  $R\Hom({\mathcal O}(1),\F)$ consists of odd representations, and so must be zero.

Now assume the conclusion of the lemma holds for a height function $h$. We will show that this implies the lemma for $\sigma_{i}^{-}h$, where $i$ is a source in $h$. Suppose that $R\Hom_{G}(F_{k}^{\sigma_{i}^{-}h},\F)=0$ for all $k$. We want to see that this implies $\F=0$. Since $\sigma_{i}^{-}h$ differs from $h$ only at $i$, we have $R\Hom_{G}(F_{k}^{\sigma_{i}^{-}h},\F)=0$ for all $k \neq i$ by assumption on $h$. We claim further that $R\Hom_{G}(F_{i}^{h},\F)=0$, and so we shall have $\F=0$ by the assumption on $h$. To sustain the claim, recall sequence \ref{euler}: $0 \rightarrow F_{i}^{\sigma_{i}^{-}h} \rightarrow \bigoplus_{i \rightarrow j}F_{j}^{h} \rightarrow F_{i}^{h} \rightarrow 0$.
Applying $R\Hom_{G}(-,\F)$ gives an exact triangle of complexes of vector spaces
$$R\Hom_{G}(F_{i}^{h},\F) \rightarrow \bigoplus_{i \rightarrow j} R\Hom_{G}(F_{j}^{h},\F) \rightarrow R\Hom_{G}(F_{i}^{\sigma_{i}^{-}h},\F).$$
The last two terms are zero by assumption on $\sigma_{i}^{-}h$, so the first term must be zero too, as claimed. 

A similar argument shows that if the lemma holds for $h$, then it also holds for $\sigma_{i}^{+}h$ when $i$ is a sink. Thus the lemma holds for all height functions.
\end{proof}

We conclude this section by making some standard remarks about categories of modules over finite dimensional algebras (see \cite{ars}) and introducing some important $t$-structures on the category $D^{b}_{\widetilde{G}}(\mathbb{P}^1)$. 

Since the algebra $A_{h}=\End_{G}(\oplus_{i} F_{i}^{h})^{\rm op}$ is finite dimensional, the category $A_{h}\mbox{-mod}$ of finitely generated modules is of finite length, meaning that every object has a finite filtration with simple quotients, and by the Jordan-Holder theorem these simples and their multiplicities do not depend on the filtration. The simple representations of the algebra $A_{h}$ are indexed by the vertices of the diagram $\Gamma$. Given a vertex $i$,  we have the $i$th idempotent $e_{i}=1 \in \Homg(F_{i}^{h},F_{i}^{h})$ and
the corresponding simple is 
$$S_{i}^{h}:=e_{i}A_{h}e_{i} = e_{i} \Homg(\oplus_{j}F_{j}^{h},\oplus_{j}F_{j}^{h})=\Homg(F_{i}^{h},F_{i}^{h}).$$
In terms of the quiver $Q_{h}$, the simple consists of a one-dimensional vector space at $i$ and zeroes elsewhere. By the Jordan-Holder 
theorem, the classes of the simples $S_{i}^{h}$ form a basis for $K_{0}(Q_{h})$.

We also have the indecomposable projectives $P_{i}^{h}=A_{h}e_{i}$, which are dual to the $S_{i}^{h}$ under the Euler form on $K_{0}(Q_{h})$:
$$\langle P_{i}^{h},S_{j}^{h} \rangle= \sum_{k}(-1)^{k}{\rm dim} \, \Ext^{k}(P_{i}^{h},S_{j}^{h})={\rm dim} \, \Hom(P_{i}^{h},S_{j}^{h})=\delta_{ij}.$$
Since every representation has a resolution by sums of the $P_{i}^{h}$, the classes of the $P_{i}^{h}$ span $K_{0}(Q_{h})$, and by duality with $S_{i}^{h}$, they are linearly independent, so the indecomposable projectives provide another basis for $K_{0}(Q_{h})$.

Applying the inverse equivalences from Theorem~\ref{kir}, we get for each height function $h$ the heart of a bounded $t$-structure 
$${\mathcal A}_{h}:=R\Phi_{h}^{-1}(A_{h}\mbox{-mod}) \subset D^{b}_{\widetilde{G}}(\mathbb{P}^1)$$
with simple objects $E_{i}^{h}=R\Phi_{h}^{-1}(S_{i}^{h})$ and indecomposable projectives $R\Phi_{h}^{-1}(P_{i}^{h})$, which are dual with respect to the Euler form on $K_{0}(\D^{b}_{\widetilde{G}}(\mathbb{P}^1))$. In fact, $R\Phi_{h}(F_{i}^{h})=\Homg(\oplus_{j}F_{j}^{h},F_{i}^{h})=A_{h}e_{i}=P_{i}^{h}$, so the indecomposable projectives of $\mathcal{A}_{h}$ are just the objects $F_{i}^{h}=R\Phi_{h}^{-1}(P_{i}^{h})$.

\begin{rem}\label{character}
Notice that since $R\Phi_{h}(E_{i}^{h})=R\Homg(\oplus_{j}F_{j}^{h}, E_{i}^{h})= S_{i}^{h}$, which is just a one-dimensional vector space concentrated at the $i$th vertex of $Q_{h}$, the simple $E_{i}^{h} \in {\mathcal A}_{h}$ is the unique object in $D^{b}_{\widetilde{G}}(\mathbb{P}^1)$ for which 
\begin{displaymath} 
R\Homg(F_{k}^{h},E_{i}^{h}) =\left\{ \begin{array}{ll} 
\mathbb{C} & \textrm{if $i=k$}\\ 
0 & \textrm{if $i \neq k$} 
\end{array} \right. 
\end{displaymath} 
where $\mathbb{C}$ denotes the complex with $\mathbb{C}$ in degree zero and zeroes elsewhere.
\end{rem}

\medskip

The hearts ${\mathcal A}_{h}$ will play an important role in our discussion. In particular, we need the following two lemmas, which describe how the simples of one heart are related to each other and to the simples of another heart.

\begin{lemma}\label{simpext} We have

\begin{displaymath} 
{\rm dim} \; \Extg^{k}(E_{i}^{h}, E_{j}^{h}) = \left\{ \begin{array}{ll} 
1 & \textrm{if $i=j$ and $k=0$}\\ 
1 & \textrm{if $i \rightarrow j$ and $k=1$}\\  
0 & \textrm{otherwise}
\end{array} \right. 
\end{displaymath} 
\end{lemma}

\begin{proof}
The corresponding statement can be easily checked for the simples $S_{i}^{h}$ (just think about when there can be morphisms and extensions between the simple representations of $Q_{h}$), and so the lemma follows upon applying the inverse equivalence $R\Phi_{h}^{-1}$.
\end{proof}

\begin{lemma}\label{simpsdiffer}
If $i$ is a source, the simples of ${\mathcal A}_{\sigma_{i}^{-}h}$ are given by
\begin{displaymath}
E_{j}^{\sigma_{i}^{-}h} = \left\{ \begin{array}{ll} 
E_{i}^{h}[-1] & \textrm{if $i=j$}\\ 
E_{j}^{h} & \textrm{if $i \neq j, i \nrightarrow j$}\\  
{\rm Cone}(E_{i}^{h}[-1] \rightarrow E_{j}^{h}) & \textrm{if $i \rightarrow j$}
\end{array} \right.
\end{displaymath} 
If $i$ is a sink, the simples of ${\mathcal A}_{\sigma_{i}^{+}h}$ are
\begin{displaymath}
E_{j}^{\sigma_{i}^{+}h} = \left\{ \begin{array}{ll} 
E_{i}^{h}[1] & \textrm{if $i=j$}\\ 
E_{j}^{h} & \textrm{if $i \neq j, j \nrightarrow i$}\\  
{\rm Cone}(E_{j}^{h}[-1] \rightarrow E_{i}^{h}) & \textrm{if $j \rightarrow i$}
\end{array} \right.
\end{displaymath} 
Here, when $i \rightarrow j$, $E_{i}^{h}[-1] \rightarrow E_{j}^{h}$ is the non-zero morphism, unique up to scalar, provided by Lemma~\ref{simpext}, and likewise for $E_{j}^{h}[-1] \rightarrow E_{i}^{h}$ when $j \rightarrow i$.

\end{lemma}

\begin{proof}
We give the proof for ${\mathcal A}_{\sigma_{i}^{-}h}$. The argument for
${\mathcal A}_{\sigma_{i}^{+}h}$ is similar.

\medskip

\noindent Claim 1: $E_{i}^{\sigma_{i}^{-}h} = E_{i}^{h}[-1]$.

Up to isomorphism, $E_{i}^{\sigma_{i}^{-}h}$ is the unique object of $D^{b}_{\widetilde{G}}(\mathbb{P}^1)$
such that $R\Homg(F_{k}^{\sigma_{i}^{-}h},E_{i}^{\sigma_{i}^{-}h})=0$ for $k \neq i$ and $R\Homg(F_{i}^{\sigma_{i}^{-}h},E_{i}^{\sigma_{i}^{-}h})\simeq \mathbb{C}$ (see Remark~\ref{character}), so to establish the claim we check that $E_{i}^{h}[-1]$ satisfies these two conditions. For the first condition, note that $R\Homg(F_{k}^{\sigma_{i}^{-}h},E_{i}^{h}[-1]) \simeq
R\Homg(F_{k}^{h},E_{i}^{h})[-1]=0$, since $F_{k}^{\sigma_{i}^{-}h}\simeq F_{k}^{h}$ for $k \neq i$. For the second condition, we need $R\Homg(F_{i}^{\sigma_{i}^{-}h},E_{i}^{h}[-1])\simeq \mathbb{C}$. Applying $R\Homg(-,E_{i}^{h})$ to the sequence $0 \rightarrow F_{i}^{\sigma_{i}^{-}h} \rightarrow \bigoplus_{i \rightarrow k}F_{k}^{h} \rightarrow F_{i}^{h} \rightarrow 0$, we get an exact triangle $$R\Homg(F_{i}^{h},E_{i}^{h}) \rightarrow \bigoplus_{i \rightarrow k}R\Homg(F_{k}^{h},E_{i}^{h}) \rightarrow R\Homg(F_{i}^{\sigma_{i}^{-}h},E_{i}^{h}).$$
Since $R\Homg(F_{i}^{h},E_{i}^{h}) \simeq \mathbb{C}$ and $\bigoplus_{i \rightarrow k}R\Homg(F_{k}^{h},E_{i}^{h})=0$, we have the desired isomorphism $R\Homg(F_{i}^{\sigma_{i}^{-}h},E_{i}^{h}[-1]) \simeq R\Homg(F_{i}^{h},E_{i}^{h}) \simeq \mathbb{C}$. 

\medskip

\noindent Claim 2: $E_{j}^{\sigma_{i}^{-}h}=E_{j}^{h}$ if $i \neq j, i \nrightarrow j$.

As in Claim 1, we check that $E_{j}^{h}$ satisfies the characteristic properties of $E_{j}^{\sigma_{i}^{-}h}$. Since $F_{k}^{\sigma_{i}^{-}h}=F_{k}^{h}$ for $k \neq i$, note that $R\Homg(F_{j}^{\sigma_{i}^{-}h}, E_{j}^{h}) \simeq \mathbb{C}$ and $R\Homg(F_{k}^{\sigma_{i}^{-}h}, E_{j}^{h})=0$ for $k \neq i, j$. It remains to show that $R\Homg(F_{i}^{\sigma_{i}^{-}h}, E_{j}^{h})=0$. Applying $R\Homg(-, E_{j}^{h})$ to sequence \ref{euler} gives us an exact triangle 
 $R\Homg(F_{i}^{h},E_{j}^{h}) \rightarrow \bigoplus_{i \rightarrow k}R\Homg(F_{k}^{h},E_{j}^{h}) \rightarrow R\Homg(F_{i}^{\sigma_{i}^{-}h},E_{j}^{h})$,
whose first two terms and hence last term are zero. 

\medskip

\noindent Claim 3: $E_{j}^{\sigma_{i}^{-}h}={\rm Cone}(E_{i}^{h}[-1] \rightarrow E_{j}^{h})$ if $i \rightarrow j$.

 Let $X={\rm Cone}(E_{i}^{h}[-1] \rightarrow E_{j}^{h})$. We check that $X$ satisfies the characteristic properties of $E_{j}^{\sigma_{i}^{-}h}$. Applying $R\Homg(F_{k}^{\sigma_{i}^{-}h}, -)$ to the exact triangle $E_{i}^{h}[-1] \rightarrow E_{j}^{h} \rightarrow X$ gives the exact triangle 

\begin{equation}\label{charactertriangle}
 \;\;\;\; R\Homg(F_{k}^{\sigma_{i}^{-}h}, E_{i}^{h}[-1]) \rightarrow R\Homg(F_{k}^{\sigma_{i}^{-}h},E_{j}^{h}) \rightarrow R\Homg(F_{k}^{\sigma_{i}^{-}h}, X).
\end{equation}
If $k \neq i,j$, then $F_{k}^{\sigma_{i}^{-}h}=F_{k}^{h}$, the first two terms and hence the last term of the triangle are zero.
if $k=j$, then $F_{k}^{\sigma_{i}^{-}h}=F_{j}^{h}$ and so the first term is zero, giving an isomorphism $ R\Homg(F_{j}^{\sigma_{i}^{-}h}, X) \simeq R\Homg(F_{j}^{h},E_{j}^{h}) \simeq \mathbb{C}$. Finally,
if $k=i$, then applying $R\Homg(-,E_{j}^{h})$ and $R\Homg(-,E_{i}^{h})$ to the sequence \ref{euler} shows that $R\Homg(F_{i}^{\sigma_{i}^{-}h}, E_{i}^{h}[-1]) \simeq \mathbb{C}$ and
$R\Homg(F_{i}^{\sigma_{i}^{-}h},E_{j}^{h}) \simeq \mathbb{C}$ (both concentrated in degree $0$). Thus if the first arrow in the triangle is non-zero,
it must give an isomorphism $R\Homg(F_{i}^{\sigma_{i}^{-}h}, E_{i}^{h}[-1]) \simeq R\Homg(F_{i}^{\sigma_{i}^{-}h},E_{j}^{h})$ and so $R\Homg(F_{k}^{\sigma_{i}^{-}h}, X)=0$, and we shall have verified that $X={\rm Cone}(E_{i}^{h}[-1] \rightarrow E_{j}^{h})$ satisfies the characteristic properties of $E_{j}^{\sigma_{i}^{-}h}$. 

That the first arrow in the triangle \ref{charactertriangle} is indeed non-zero follows from the fact that $\oplus_{k} F_{k}^{\sigma_{i}^{-}h}$ generates the derived category. Indeed,
apply the functor $R\Homg(\oplus_{k} F_{k}^{\sigma_{i}^{-}h}, -)$ to the morphism $E_{i}^{h}[-1] \rightarrow E_{j}^{h}$. We have seen in the course of
our argument that the result is zero when restricted to summands with $k \neq i$. If it were also zero when $k=i$, then the morphism $E_{i}^{h}[-1] \rightarrow E_{j}^{h}$ would have to be zero by faithfulness of $R\Homg(\oplus_{k} F_{k}^{\sigma_{i}^{-}h}, -)$, a contradiction.
\end{proof}


\section{McKay correspondence for $\T$}\label{sectcotanmckay}
We now give the analogue for the cotangent bundle $\T$ of the equivalences $R\Phi_{h}: D^{b}_{\widetilde{G}}(\mathbb{P}^1) \simeq D^{b}(A_{h})$ from Theorem~\ref{kir}.

\medskip
Let $\pi$ be the projection and $s$ the zero-section for $\T$:
\begin{equation*}
\xymatrix{
\T \ar[d]_{\pi}\\
\mathbb{P}^1 \ar@/^{2pc}/[u]^{\quad}_{s}
}
\end{equation*}

Define $\F_{i}^{h}:=\pi^{\ast} F_{i}^{h}$ and $\E_{i}^{h}:=s_{*}E_{i}^{h}$, analogues for $\T$ of the indecomposable projectives and the simples for the heart ${\mathcal A}_{h} \subset D^{b}_{\widetilde{G}}(\mathbb{P}^1)$. Setting $B_{h}:=\End_{G}(\oplus_{i} \F_{i}^{h})^{\rm op}$, consider the natural functor $R\Psi_{h}=R\Homg(\oplus_{i}\F_{i}^{h}, -)$ from the derived category of $\widetilde{G}$-sheaves on $\T$ to the derived category of finitely generated $B_{h}$-modules.

\medskip
The following theorem is analogous to \cite[Proposition 4.1]{tstruc}
\medskip

\begin{theorem} \label{cotan}
For each height function $h$, we have an equivalence

\begin{equation*} 
\xymatrix{
D^{b}_{\widetilde{G}}(\T) \ar[r]^{R\Psi_{h}} & D^{b}(B_{h}). 
}
\end{equation*}
\end{theorem}

\begin{proof}
Like in the proof of \ref{equiv}, we must check that there are no higher $\Ext$s between the $\F_{i}^{h}$  and that the $\F_{i}^{h}$ generate $D^{b}_{\widetilde{G}}(\T)$. This will establish that $R\Psi_{h}$ gives an equivalence $D^{b}_{G}(\T) \simeq {\rm perf}\; B_{h}$.  In Proposition~\ref{cokoszul}, we shall see that $B_{h}$ is Koszul and hence of finite global dimension, so ${\rm perf}\; B_{h} \simeq D^{b}(B_{h})$.

\medskip

To compute $\Ext$s, use the adjunction $\pi^{\ast} \dashv \pi_{*}$ and the projection formula: 

\begin{align*}
\Extg^{k}(\F_{i}^{h}, \F_{j}^{h}) & = \Extg^{k}( \pi^{\ast}F_{i}^{h},\pi^{\ast}F_{j}^{h}) \simeq \\
\Extg^{k}(F_{i}^{h},\pi_{\ast}\pi^{\ast}F_{j}^{h}) & \simeq \bigoplus_{d \geq 0}\Extg^{k}(F_{i}^{h},\oplus_{k} F_{j}^{h} \otimes T^{\otimes d}),
\end{align*}
where $T$ denotes the tangent bundle of $\mathbb{P}^1$. Each summand on the right is zero by Lemma~\ref{exts}, so indeed
we have vanishing of the higher $\Ext$s.

Next we establish generating. Suppose that we have $\G \in D^{b}_{\widetilde{G}}(\T)$ such that $R\Homg(\pi^{\ast} F_{i}^{h}, \G)=0$ for all $i$. Applying the adjunction, we have 
$R\Homg(F_{i}^{h}, R\pi_{\ast} \G)=0$ for all $i$, so by Lemma~\ref{span} above, $R\pi_{\ast} \G=0$. But $\pi$ is an affine map, so $\pi_{\ast}$ is exact and has no kernel, hence $\G=0$. 
\end{proof}

\begin{rem}
Note that our algebra $B_{h} \simeq \bigoplus_{i,j,d}\Homg(F_{i}^{h},F_{j}^{h}\otimes T^{\otimes d})^{\rm op}$ is graded by the difference $h(j) + 2d - h(i)$. Since $F_{i}^{h} = \mathcal{O}(h_{i}) \otimes W_{i}$
and $F_{j}^{h}\otimes T^{\otimes d} \simeq \mathcal{O}(h(j) + 2d) \otimes W_{j}$, there is an isomorphism $\Homg(F_{i}^{h},F_{j}^{h}\otimes T^{\otimes d}) \simeq \Homg(W_{i}, \mathcal{O}(h(j) + 2d - h(j)) \otimes W_{j})$, so
the degree zero part of the algebra $B_{h}$ is just $B_{0}=\bigoplus_{i} \Homg(F_{i}^{h}, F_{i}^{h})$,
which is a commutative semisimple $\mathbb{C}$-algebra with one summand for each $i$. 
\end{rem}

As in Section~\ref{affinecase}, we shall apply Koszul duality to understand the graded algebra $B_{h}$. For this,
we need to compute some $\Ext$s, which we shall do using the following lemma from \cite[pg. 20]{tstruc}.

\begin{lemma}\label{push}
For $\F, \G \in D^{b}_{\widetilde{G}}(\mathbb{P}^1)$ we have 
\begin{equation*}
\Extg^{k}(s_{*}\F,s_{*}\G)\simeq \Extg^{k}(\F,\G)\oplus\Extg^{2-k}(\G,\F)^{*}
\end{equation*}
\end{lemma}

In particular, the lemma allows us to compute $\Ext$s between the objects $\E_{i}^{h}=s_{\ast}E_{i}^{h}$.

\begin{propo}\label{spher}
Let $h$ be a height function on $\Gamma$ and set $\E_{i}^{h}=s_{\ast}E_{i}^{h}$. Then we have
\begin{align*}
\Homg(\E_{i}^{h},\E_{j}^{h}) & \simeq \Homg(E_{i}^{h},E_{j}^{h})\\ 
\Extg^{1}(\E_{i}^{h},\E_{j}^{h}) & \simeq \Extg^{1}(E_{i}^{h},E_{j}^{h}) \oplus \Extg^{1}(E_{j}^{h},E_{i}^{h})^{*}\\
\Extg^{2}(\E_{i}^{h},\E_{j}^{h})  & \simeq \Extg^{2}(E_{j}^{h},E_{i}^{h})^{*}
\end{align*}
For any height function $h$, the $\E_{i}^{h}$ form a $\Gamma$-configuration of spherical objects. \end{propo}

\begin{proof}
The three isomorphisms are just Lemma~\ref{push}. 

To see that the $\E_{i}^{h}$ form a $\Gamma$-configuration, note that $\Homg(\E_{i}^{h},\E_{i}^{h}) \simeq \Extg^{2}(\E_{i}^{h},\E_{i}^{h})^{*} \simeq \mathbb{C}$ and $\Extg^{1}(\E_{i}^{h},\E_{i}^{h}) =0$ by Lemma~\ref{simpext} together with the three isomorphisms. Thus $\E_{i}^{h}$ is indeed spherical. One can see in the same way that $\Extg^{1}(\E_{i}^{h},\E_{j}^{h}) \simeq \mathbb{C}$ exactly when $i$ and $j$ are connected in $Q$ and that $\Homg(\E_{i}^{h},\E_{j}^{h}) = \Extg^{2}(\E_{i}^{h},\E_{j}^{h})=0$ when $i \neq j$.
\end{proof}


\begin{propo}\label{cokoszul}
$B_{h}$ is Koszul with Koszul dual $E(B_{h}) \simeq \Extg^{\bullet}(\oplus_{i} \E_{i}^{h},\oplus_{i} \E_{i}^{h})$, the $\Ext$ algebra of the spherical $\Gamma$-collection $\E_{i}^{h}$. Thus by Lemma~\ref{lemmagammaconfigs}, there is an isomorphism $B_{h} \simeq \Pi_{\Gamma}$ with the preprojective algebra of $\Gamma$.
\end{propo}

\begin{proof}
From the adjunction $s^{*} \dashv Ls_{*}$, we see that 
$$R\Homg(\pi^{*}F_{i}^{h}, s_{*}E_{j}^{h}) \simeq R\Homg(Ls^{*}\pi^{*}F_{i}^{h}, E_{j}^{h}) \simeq
R\Homg(F_{i}^{h}, E_{j}^{h}).$$
By the remarks before Lemma~\ref{simpext}, the right hand side is zero when $i \neq j$ and is $1$-dimensional and concentrated in degree
zero when $i=j$. Thus we see that $R\Psi_{h}(\E_{i}^{h}) \simeq e_{i}B_{h}e_{i}$, the $i$th simple of the algebra $B_{h}$.  Since $R\Psi_{h}$ is an equivalence, 

$$E(B_{h})=\Ext^{\bull}_{B_{h}}(B_{0},B_{0}) \simeq \Extg^{\bull}(\oplus_{i} \E_{i}^{h}, \oplus_{i} \E_{i}^{h}).$$  

Thus $E(B_{h})$ is the $\Ext$-algebra of a $\Gamma$-configuration
and by Lemma~\ref{lemmagammaconfigs}, there is an isomorphism $B_{h} \simeq \Pi_{\Gamma}$ and $B_{h}$ is Koszul.
\end{proof}

\begin{rem}
Putting together the equivalences of Theorems~\ref{cotan} and \ref{equiv} and the isomorphisms $B_{h} \simeq \Pi_{\Gamma}$, we see that for each height function $h$ there is  a chain of equivalences
$$D^{b}_{\widetilde{G}}(\T) \simeq D^{b}(B_{h}) \simeq D^{b}(\Pi_{\Gamma}) \simeq D^{b}_{G}(\mathbb{C}^{2}),$$
which provides a bridge between the projective McKay correspondence of \cite{kir} and the usual McKay correspondence for $\mathbb{C}^2$.
\medskip

As pointed out by Khovanov-Huerfano \cite{catadj}, a single equivalence $D^{b}_{\widetilde{G}}(\T) \simeq D^{b}_{G}(\mathbb{C}^{2})$ can be obtained by noting that there is an isomorphism of resolutions $\widehat{Y} \rightarrow \T/{\widetilde{G}}$ and $\widehat{X} \rightarrow {\mathbb C}^{2}/G$. Applying the celebrated theorem of Bridgeland-King-Reid \cite{bkr} then gives equivalences
$$D^{b}_{\widetilde{G}}(\T)\simeq D^{b}(Y) \simeq D^{b}(Y^{\prime}) \simeq D^{b}_{G}(\mathbb{C}^{2}).$$
\end{rem}


\section{Reflection functors and spherical twists}\label{sectreflectfunctors}

One of the most interesting aspects of Kirillov's paper \cite{kir} is that the equivalences $R\Phi_{h}$ for different $h$ are related by the reflection functors of Bernstein-Gelfand-Ponamarev \cite{bgp}.
We show that in terms of $D^{b}_{\widetilde{G}}(\mathbb{P}^{1})$, the reflection functors amount to tilting at a simple object. 
On $D^{b}_{\widetilde{G}}(\T)$ the reflection functors are replaced by spherical twists which relate the various equivalences
$R\Psi_{h}$. We also note that the action of the twist can be described in terms of tilting at a simple object. This completes our description of the relation betweeen the McKay correspondences for $\T$ and $\mathbb{P}^1$ as outlined in the table from the introduction.


\subsection{Reflection functors}

Recall from Section~\ref{sectprojmckay} that if $i$ is a sink in a quiver $Q$,
we define a new quiver $\sigma_{i}^{+}Q$ by reversing all arrows adjacent to $i$ so that it becomes a source. Likewise, if $i$ is a source, we
define $\sigma_{i}^{-}Q$ so that $i$ becomes a sink.

Accompanying these operations on quivers are the {\it reflection functors} of Bernstein-Gelfand-Ponomarev \cite{bgp}
\begin{equation*}
\xymatrix{
{\rm Rep}\, Q \ar[r]^{\sigma_{i}^{+}} & {\rm Rep}\, {\sigma_{i}^{+}}Q &  {\rm and} & {\rm Rep}\, Q \ar[r]^{\sigma_{i}^{-}} & {\rm Rep}\, {\sigma_{i}^{-}}Q.
}
\end{equation*}
In the first case, given a sink $i$ in $Q$ and representation $V$, define $\sigma_{i}^{+}V$ to be the same as $V$ away from $i$, and at $i$ replace $V_{i}$
with the kernel of the natural morphism $\bigoplus_{j \rightarrow i} V_{j} \rightarrow V_{i}$. The arrows from $(\sigma_{i}^{+}V)_{i}$ to adjacent $V_{j}$ are given by the composition $(\sigma_{i}^{+}V)_{i} \hookrightarrow  \bigoplus_{j \rightarrow i} V_{j} \rightarrow V_{j}$. This defines the functor on objects and its definition on morphisms is the obvious one. Likewise, if $i$ is a source, $\sigma_{i}^{-}$ does nothing away from $i$, and at $i$ replace $V_{i}$ with the cokernel of the morphism
$V_{i} \rightarrow \bigoplus_{i \rightarrow j} V_{j}$. The arrows from adjacent $V_{j}$ to $(\sigma_{i}^{-}V)_{i}$ are given by the composition
$V_{j} \hookrightarrow \bigoplus_{i \rightarrow j} V_{j} \rightarrow (\sigma_{i}^{-}V)_{i}$ and the definition of the functor on morphisms is obvious.

We record some basic and well-known facts about the reflection functors.

\begin{lemma} 
\begin{enumerate}
\item The functor $\sigma_{i}^{+}$ is left exact, while $\sigma_{i}^{-}$ is right exact, and we have an adjunction $\sigma_{i}^{-} \dashv \sigma_{i}^{+}$.
\item The derived functors $R\sigma_{i}^{+}$ and $L\sigma_{i}^{-}$ are inverse equivalences. Identifying the Grothendieck groups of $Q$ and $\sigma_{i}^{\pm}Q$ using the bases of simple representations, the automorphisms of the Grothendieck group induced by $R\sigma_{i}^{+}$ and $L\sigma_{i}^{-}$ are simply reflections at the $i$th simple.
\end{enumerate}
\end{lemma}

In the case of a Dynkin diagram, the functors thus generate the action of the Weyl group on the root lattice, which we identify with $K_{0}$ of the quiver.

Theorem 8.9 in \cite{kir} gives the relation between the equivalences $R\Phi_{h}$ for different height functions in terms of reflection functors.


\begin{theorem}\label{reflect}
We have a commutative diagram of equivalences
\begin{equation*}
\xymatrix{& D^{b}_{\widetilde{G}}(\mathbb{P}^{1}) \ar[dl]_{R\Phi_{h}} \ar[dr]^{R\Phi_{\sigma_{i}^{-}h}}\\
D^{b}(Q_{h}) \ar[rr]^{L\sigma_{i}^{-}} && D^{b}(Q_{\sigma_{i}^{-}h}).
}
\end{equation*}
Likewise, we have $R\sigma_{i}^{+}\circ R\Phi_{h} \simeq R\Phi_{\sigma_{i}^{+}h}$.
\end{theorem}

As discussed in the comments before Lemma~\ref{simpext}, applying the inverse equivalence $R\Phi_{h}^{-1}$ to the standard heart of $D^{b}(Q_{h})$ gives a non-standard heart ${\mathcal A}_{h} \subset  D^{b}_{\widetilde{G}}(\mathbb{P}^{1})$ of finite length with simples $E_{i}^{h}$ and indecomposable projectives $F_{i}^{h}$. 

To relate the various hearts ${\mathcal A}_{h}$ we  use the following proposition from \cite{tstruc}.


\begin{propo} Let ${\mathcal A} \subset \D$ be a finite length heart of a bounded $t$-structure for $\D$ and let $S \in {\mathcal A}$ be a simple object,
Set $\langle S \rangle^{\perp}=\{\F \in \mathcal A \;| \; \Hom_{\mathcal A}(S, \F)=0 \}$ and $^{\perp}\!\langle S \rangle=\{\F \in \mathcal A \;| \; \Hom_{\mathcal A}(\F,S)=0 \}$. Then the full subcategories
$$L_{S}{\mathcal A}=\{\F \in {\mathcal D} \; | \; H^{i}(\F) = 0 \mbox{ for } i \neq 0,1, H^{0} \in \langle S \rangle^{\perp} \mbox{ and } H^{1}(\F) \in  \langle S \rangle \}$$
and
$$R_{S}{\mathcal A}=\{\F \in {\mathcal D} \; | \; H^{i}(\F) = 0 \mbox{ for } i \neq -1,0, H^{-1} \in \langle S \rangle \mbox{ and } H^{0}(\F) \in  ^{\perp}\!\!\langle S \rangle \}$$
are hearts of bounded $t$-structures on $\D$. $L_{S}{\mathcal A}$ is called the {\it left tilt} at $S$, $R_{S}{\mathcal A}$ the {\it right tilt}.
\end{propo}

We can now state the relation between various hearts ${\mathcal A}_{h}$ in terms of tilting.


\begin{propo}\label{projtilt} Denoting the left and right tilts at $E_{i}^{h} \in {\mathcal A}_{h} \subset D^{b}_{\widetilde{G}}(\mathbb{P}^{1})$ by $L_{i}$ and $R_{i}$,
we have

$$L_{i}{\mathcal A}_{h}={\mathcal A}_{\sigma_{i}^{-1}h} \;\;\;\;\mbox{and}\;\;\;\; R_{i}{\mathcal A}_{h}={\mathcal A}_{\sigma_{i}^{+}h}.$$
\end{propo}

\begin{proof}
This follows essentially from the well-known relation between the reflection functors and tilting (in fact tilting was invented to generalize the reflection functors). Letting ${\mathcal C}_{h} \subset D^{b}(Q_{h})$ denote the standard heart, the relation is that $L\sigma_{i}^{-}({\mathcal C}_{\sigma_{i}^{+}h})=R_{i}{\mathcal C}_{h}$ and $R\sigma_{i}^{+}({\mathcal C}_{\sigma_{i}^{-}h})=L_{i}{\mathcal C}_{h}$. We check the first and the second is similar.

Since both  $L\sigma_{i}^{-}({\mathcal C}_{\sigma_{i}^{+}h})$ and $R_{i}{\mathcal C}_{h}$ form hearts of bounded $t$-structures for $D^{b}(Q_{h})$ and nested hearts are in fact equal, it is enough to see that $L\sigma_{i}^{-}({\mathcal C}_{\sigma_{i}^{+}h}) \subseteq R_{i}{\mathcal C}_{h}$. In fact, $L\sigma_{i}^{-}({\mathcal C}_{\sigma_{i}^{+}h})$ is finite length and is the smallest extension closed subcategory containing
its simples, so it is enough to check that $L\sigma_{i}^{-}(S_{j}^{\sigma_{i}^{+}h}) \in  R_{i}{\mathcal C}_{h}$ for every simple $S_{j}^{\sigma_{i}^{+}h} \in {\mathcal C}_{\sigma_{i}^{+}h}$.

By the definition of right tilting, we then must see that $H^{0}(L\sigma_{i}^{-}(S_{j}^{\sigma_{i}^{+}h}))=\sigma_{i}^{-}(S_{j}^{\sigma_{i}^{+}h}) \in {}^{\perp}\langle S_{i}^{h}\rangle$ and $H^{-1}(L\sigma_{i}^{-}(S_{j}^{\sigma_{i}^{+}h})) \in \langle S_{i}^{h} \rangle$. 
First note that
\begin{displaymath}
\sigma_{i}^{-}(S_{j}^{\sigma_{i}^{+}h})=\left \{ \begin{array}{ll} 
S_{j}^{h} & {\rm if} \; i \nrightarrow j \\ 
W & {\rm if} \; i \rightarrow j \\
0  & {\rm if} \; i=j
\end{array} \right.
\end{displaymath}  
where $W$ is the quiver representation with $\mathbb{C}$ at $i$ and $j$ and an isomorphism for the arrow joining them. Thus in all cases
$H^{0}(L\sigma_{i}^{-}(S_{j}^{\sigma_{i}^{+}h}))=\sigma_{i}^{-}(S_{j}^{\sigma_{i}^{+}h}) \in {}^{\perp}\langle S_{i}^{h}\rangle$.

For $H^{-1}$, consider 
a projective resolution $0 \rightarrow P^{-1} \rightarrow P^{0} \rightarrow S_{j}^{\sigma_{i}^{+}h} \rightarrow 0$. After applying the functor, the map $\sigma_{i}^{-}P^{-1} \rightarrow \sigma_{i}^{-}P^{0}$ is still injective, except possibly at $i$, so $H^{-1}(L\sigma_{i}^{-}(S_{j}^{\sigma_{i}^{+}h})) \in \langle S_{i}^{h} \rangle$.

Since tilting commutes with equivalences and we have $R\Phi_{h}^{-1} \circ L\sigma_{i}^{-} \simeq R\Phi_{\sigma_{i}^{+}h}^{-1}$ and $R\Phi_{h}^{-1} \circ R\sigma_{i}^{+} \simeq R\Phi_{\sigma_{i}^{-}h}^{-1}$,
the proposition follows.
\end{proof}


\subsection{Spherical twists}\label{sectsphericaltwists}

In the category $D^{b}_{\widetilde{G}}(\T)$ the role of the objects $E_{i}^{h} \in D^{b}_{\widetilde{G}}(\mathbb{P}^{1})$ is played by the spherical objects $\E_{i}^{h}=s_{*}E_{i}^{h}$.
In fact, for this section it is more convenient to think of the $\E_{i}^{h}$ as objects in the full triangulated subcategory $$\D \subset D^{b}_{\widetilde{G}}(\T)$$ consisting of objects (set-theoretically) supported along the zero-section. 

There are two main advantages of working with $\D$. First, it is $2$-CY, since the canonical bundle is trivial on $\T$ and the condition of support along the zero-section ensures the category is $\Hom$-finite. Second, $\D$ is 
the smallest triangulated subcategory of $D^{b}_{\widetilde{G}}(\T)$ containing the objects $\E_{i}^{h}$, $i \in I$, since any object $\F \in \D$ is an extension of its cohomology objects, each of which is set-theoretically supported on the zero-section and so has a natural filtration for which the associated graded pieces are pushed-forward from objects in $D^{b}_{\tilde{G}}(\mathbb{P}^{1})$, which are in turn extensions of the $E_{i}^{h}$, $i \in I$. 

By Proposition~\ref{spher}, the spherical objects $\E_{i}^{h}$ form a $\Gamma$-configuration. The spherical twists $T_{\E_{i}^{h}}$ and generate
an action of the braid group $B_{\Gamma}$ on $\D$ \cite{huy}.

We saw above that if we identify all of the Grothendieck groups $K_{0}(Q_{h})$ for different $h$ with the affine root lattice associated to $\Gamma$, then the reflection functors induce the action of the Weyl group. Now we want to see that $K_{0}(\D)$ can be identified with the affine root lattice and that the action of the braid group $B_{\Gamma}$ induces that of the Weyl group.



\begin{propo}\label{groth}
The classes of the $\F_{i}^{h}$ and $\E_{j}^{h}$ form dual bases with respect to the natural pairing $\langle \; , \; \rangle: K_{0}(D^{b}_{\widetilde{G}}(\T)) \otimes K_{0}(\D) \rightarrow \mathbb{Z}$, where $\langle \E, \F \rangle = \sum_{k} (-1)^{k} {\rm dim} \, \Extg^{k}(\E,\F)$. 
\end{propo}

\begin{proof}
First, recall that the pullback $\pi^{*}$ gives an isomorphism $K_{0}(D^{b}_{\widetilde{G}}(\mathbb{P}^1)) \simeq K_{0}(D^{b}_{\widetilde{G}}(\T))$ (see \cite[Theorem 5.4.17]{chriss}). Since the $F_{i}^{h}$ form a basis for the former, the $\F_{i}^{h}=\pi^{*}F_{i}^{h}$ form a basis for the latter. Since $R\Homg(\pi^{*}F_{i}^{h}, s_{*}E_{j}^{h}) \simeq R\Homg(s^{*}\pi^{*}F_{i}^{h}, E_{j}^{h}) \simeq
R\Homg(F_{i}^{h}, E_{j}^{h})$, the duality between
$\F_{i}^{h}$ and $\E_{j}^{h}$ follows from that between $F_{i}^{h}$ and $E_{j}^{h}$ discussed in Remark~\ref{character}. Then the linear independence of the $\E_{j}^{h}$ follows from the duality between the $\F_{i}^{h}$ and the $\E_{j}^{h}$.

It remains to show that the $\E_{j}^{h}$ span $K_{0}(\D)$. For this, consider an object $\F \in \D$. Its class in $K_{0}(\D)$ may be
written as $[\F]=\sum_{k}(-1)^{k}[{\mathcal H}^{k}(\F)]$, where the ${\mathcal H}^{k}(\F)$ are the cohomology sheaves of $\F$. Although the support
of $\G={\mathcal H}^{k}(\F)$ may be non-reduced and so $\G$ might not be the push-forward of an object from $\mathbb{P}^1$, there is a natural
filtration $0=\G^{m} \subseteq \dots \subseteq \G^{1} \subseteq \G$ whose associated graded pieces have reduced support along $Z$, where $\G^{k}=\sqrt{{\mathcal I}}^{k}\cdot \G$ for ${\mathcal I}={\rm Ann}(\G)$. Thus the class of $\G$ and hence the class of $\F$ can be written as a combination of
classes pushed-forward from $\mathbb{P}^1$. By the comments before Lemma~\ref{simpext}, classes on $\mathbb{P}^1$
are combinations of the classes $E_{i}^{h}$, and so $K_{0}(\D)$ is indeed spanned by the $\E_{i}^{h}$.
\end{proof}


\begin{propo}\label{cartan}
In the basis $\E_{i}^{h}$, the Euler form on $K_{0}(\D)$ is given by the Cartan matrix of $\Gamma$, so $K_{0}(\D)$ is an affine root lattice with the $\E_{i}^{h}$ as a base of simple roots. Moreover, the twists $T_{\E_{i}^{h}}$ induce the corresponding simple reflections.
\end{propo}

\begin{proof}
For the first claim, simply note that by Lemmas~\ref{simpext} and \ref{push} $\langle \E_{i}^{h}, \E_{i}^{h} \rangle = {\rm dim} \, \Homg^{k}(\E_{i}^{h},\E_{i}^{h})+ {\rm dim}\, \Extg^{2}(\E_{i}^{h},\E_{i}^{h})=2$ 
and for $i\neq j$
$\langle \E_{i}^{h}, \E_{j}^{h} \rangle = - {\rm dim} \, \Extg^{1}(\E_{i}^{h},\E_{j}^{h}) - {\rm dim} \, \Extg^{1}(\E_{j}^{h},\E_{i}^{h})^{*}= - n_{ij}$.
The second claim then follows from the expression $[T_{\E_{i}^{h}}(\E_{j}^{h})]=[\E_{j}^{h}]-\langle \E_{i}^{h}, \E_{j}^{h} \rangle [\E_{i}^{h}]$.

\end{proof}

\begin{rem}
It can be shown that while the push-forward along the zero-section gives an isomorphism $K_{0}(D^{b}_{\widetilde{G}}(\mathbb{P}^1))\simeq K_{0}(\D)$, the push-forward map $\sigma_{*}: K_{0}(D^{b}_{\widetilde{G}}(\mathbb{P}^1)) \rightarrow K_{0}(D^{b}_{\widetilde{G}}(\T))$ has kernel consisting of imaginary roots, and so the image can be thought of as a root lattice of finite type.
\end{rem}

Under the equivalence $R\Psi_{h}: D^{b}_{\widetilde{G}}(\T) \simeq D^{b}(\Pi_{\Gamma})$, the subcategory $\D$ is identified with the smallest triangulated subcategory $\D_{{\rm nil}} \subset D^{b}(\Pi_{\Gamma})$ containing the simple objects $S_{i}$. I claim that
$\D_{{\rm nil}}$ consists of complexes with nilpotent cohomology, where a $\Pi_{\Gamma}$-module $M$ is called nilpotent if there is a degree $k$ for which $\Pi_{\Gamma}^{k} \cdot M=0$. To see this, note that any nilpotent module has a filtration $$0 \subset {\Pi_{\Gamma}}^{k}\cdot M \subset {\Pi_{\Gamma}}^{k-1}\cdot M \subset \dots \subset {\Pi_{\Gamma}}^{1}\cdot M \subset M$$
whose associated graded pieces are annihilated by ${\Pi_{\Gamma}}^{1}$ and hence are sums of the simple modules $S_{i}$. It follows that the restriction of the standard heart of $D^{b}(\Pi_{\Gamma})$ to $D_{\rm{nil}}$ is of finite length. 

Under the inverse equivalence $R\Psi_{h}^{-1}$, the standard $t$-structure on $D^{b}(B_{h})$ is sent to a non-standard $t$-structure on $D^{b}_{\widetilde{G}}(\T)$, which we may restrict to $\D \subset  D^{b}_{\widetilde{G}}(\T)$. The resulting heart, which we denote ${\mathcal B}_{h} \subset \D$, is then of finite length with simple objects $\E_{i}^{h}$, $i \in I$.

\medskip

Our next result shows that the spherical twists not only realize the action of the Weyl group on the affine root lattice but also relate the various hearts ${\mathcal B}_{h} \subset \D$.


\begin{theorem}\label{twist}
If $i \in Q_{h}$
is a source, then $T_{\E_{i}^{h}}(\E_{j}^{h})\simeq \E_{j}^{\sigma_{i}^{-}h}$. Likewise, if $i$ is a sink, then $T_{\E_{i}^{h}}^{-1}(\E_{j}^{h})\simeq \E_{j}^{\sigma_{i}^{+}h}$.
In particular, since the hearts are finite length and hence determined by their simples, $T_{\E_{i}^{h}}({\mathcal B}_{h})={\mathcal B}_{\sigma_{i}^{-}h}$ for $i$ a source and $T_{\E_{i}^{h}}^{-1}({\mathcal B}_{h})={\mathcal B}_{\sigma_{i}^{+}h}$ for $i$ a sink.
\end{theorem}

\begin{proof}
We prove $T_{\E_{i}^{h}}(\E_{j}^{h})\simeq \E_{j}^{\sigma_{i}^{-}h}$ for $i$ a source. The proof of $T_{\E_{i}^{h}}^{-1}(\E_{j}^{h})\simeq \E_{j}^{\sigma_{i}^{+}h}$ is similar.

Consider the defining exact triangle $R\Homg(\E_{i}^{h},\E_{j}^{h})\otimes\E_{i}^{h} \rightarrow \E_{j}^{h} \rightarrow T_{\E_{i}^{h}}(\E_{j}^{h})$. 

If $i=j$, then $T_{\E_{i}^{h}}(\E_{i}^{h}) \simeq \E_{i}^{h}[-1] \simeq s_{*}E_{i}^{h}[-1] \simeq s_{*}E_{i}^{\sigma_{i}^{-}h} \simeq \E_{i}^{\sigma_{i}^{-}h}$ with the first isomorphism being a standard property of spherical twists
in a $2$-CY category, the second isomorphism is from exactness of $s_{*}$, the third by Lemma~\ref{simpsdiffer}, and the last by definition. 

If $i \neq j$ and $i \nrightarrow j$, then $R\Homg(\E_{i}^{h},\E_{j}^{h})=0$ because
the $\E_{k}^{h}$ form a $\Gamma$-configuration, so 
$T_{\E_{i}^{h}}(\E_{j}^{h}) \simeq \E_{j}^{h} \simeq \E_{j}^{\sigma_{i}^{-}h}$,
with the last isomorphism coming from Lemma~\ref{simpsdiffer}.

If $i \rightarrow j$, then $R\Homg(\E_{i}^{h},\E_{j}^{h}) \simeq \Extg^{1}(\E_{i}^{h},\E_{j}^{h}) \simeq \mathbb{C}$ by the properties of $\Gamma$-configurations. Thus the defining
exact triangle is of the form  $\E_{i}^{h}[-1] \rightarrow \E_{j}^{h} \rightarrow T_{\E_{i}^{h}}(\E_{j}^{h})$. But by Lemma~\ref{simpsdiffer}, $E_{j}^{\sigma_{i}^{-}h} \simeq {\rm Cone}(E_{i}^{h}[-1] \rightarrow E_{j}^{h})$, so indeed $T_{\E_{i}^{h}}(\E_{j}^{h})\simeq \E_{j}^{\sigma_{i}^{-}h}$.

\end{proof}

Note that the relation among the hearts $\mathcal{B}_{h} \subset \D$ by autoequivalences is stronger than the relation among the hearts $\mathcal{A}_{h}$ by tilting (Proposition~\ref{projtilt}). Our final result, which is well-known to experts, shows that the weaker relation of tilting is induced by the spherical twists, thus completing the analogy between the spherical twists  and the reflection functors outlined in the table from the introduction.

\begin{propo}\label{twisttilt}
Let $\D$ be a $2$-CY triangulated category with ${\mathcal B} \subset \D$ the heart of a bounded $t$-structure that is of finite length. Then twists along simple, spherical objects realize tilting at $S$:
$$T_{S}({\mathcal B})=L_{S}{\mathcal B}\;\;\;\;\mbox{and}\;\;\;\; T_{S}^{-1}({\mathcal B})=R_{S}{\mathcal B}.$$
\end{propo}

\begin{proof}
Since bounded $t$-structures with nested hearts are equal , it is enough to check that $T_{S}({\mathcal A}) \subseteq L_{S}({\mathcal A})$,
and since the $T_{S}({\mathcal A})$ is finite length, it is enough to check that $T_{S}(S^{\prime}) \in L_{S}({\mathcal A})$ for every simple $S^{\prime} \in {\mathcal A}$. 

When $S=S^{\prime}$, we know that $T_{S}(S)=S[-1]$
so that indeed $H^{0}(T_{S}(S))=0 \in \langle S \rangle^{\perp}$ and $H^{1}(T_{S}(S))=S \in \langle S \rangle$. Thus $T_{S}(S) \in L_{S}{\mathcal A}$.

Otherwise consider the exact triangle
$$R\Hom(S,S^{\prime})\otimes S \rightarrow S^{\prime} \rightarrow T_{S}(S^{\prime}).$$
By Schur's lemma, $\Hom(S, S^{\prime})=\Hom(S^{\prime},S)=0$, and so by Serre duality
$\Ext^{2}(S,S^{\prime})=0$. Then from the long exact sequence in cohomology we see that $H^{i}(T_{S}(S^{\prime}))=0$ for $i \neq 0$ so that $T_{S}(S^{\prime}) \simeq H^{0}(T_{S}(S^{\prime}))$. The non-zero part of the long exact sequence is thus
$$0 \rightarrow S^{\prime} \rightarrow T_{S}(S^{\prime}) \rightarrow \Ext^{1}(S, S^{\prime}) \otimes S \rightarrow 0.$$ Applying $\Hom(S, -)$ gives 
$$0 \rightarrow \Hom(S,T_{S}(S^{\prime})) \rightarrow \Ext^{1}(S, S^{\prime}) \otimes \Hom(S,S) \ \rightarrow \Ext^{1}(S, S^{\prime}) \rightarrow 0.$$ The map on the right being an isomorphism,
we have $\Hom(S,T_{S}(S^{\prime})) = 0$, whence $T_{S}(S^{\prime}) \in \langle S \rangle^{\perp}$.
\end{proof}

\bibliographystyle{plain}
\bibliography{projmckaybiblio}

\end{document}